\documentclass[a4paper, 12pt]{amsart}
\newcommand{\cev}[1]{\reflectbox{\ensuremath{\vec{\reflectbox{\ensuremath{#1}}}}}}
\usepackage[margin=2.4cm]{geometry}
\usepackage{amsmath}
\usepackage{amsfonts}
\usepackage{amssymb}
\usepackage{xcolor}
\usepackage{mathrsfs}
\usepackage{etoolbox}
\usepackage{array}
\AtBeginEnvironment{theorem}{\hangindent=2em}
\renewenvironment{proof}{{\bfseries Proof.}}{\qed}
\numberwithin{equation}{section} 
\newtheorem{theorem}{Theorem}[section] 
\newtheorem{pro}[theorem]{Proposition} 
 
\newtheorem{lemma}[theorem]{Lemma} 

\newtheorem{question}[theorem]{Question}
\theoremstyle{definition}
\newtheorem{definition}[theorem]{Definition}
\newtheorem{defn}[theorem]{Definition} 
\newtheorem{remark}[theorem]{Remark} 
\newtheorem{example}[theorem]{Example}

\newcommand{\ap}{\alpha}

\newcommand{\imp}{\Rightarrow}
\newcommand{\z}{\mathbb{Z}}

\newcommand{\m}{\gamma^{m}_{p}}

\newcommand{\thmref}[1]{Theorem~\ref{#1}}
\newcommand{\lemref}[1]{Lemma~\ref{#1}}

\newcommand{\propref}[1]{Proposition~\ref{#1}}

\newcommand{\rkref}[1]{Remark~\ref{#1}}
\newcommand{\bl}[1]{\textcolor{blue}{#1}}

\numberwithin{equation}{section}

\usepackage[backref]{hyperref}

\usepackage{tikz,standalone}
\usepackage{tikz-cd}
\usetikzlibrary{arrows,chains,quotes,matrix,positioning,scopes}
\usepackage{pgfplots}
\pgfplotsset{compat=1.15}
\usetikzlibrary{arrows}
\usetikzlibrary{lindenmayersystems}

\usepackage[font=small,labelfont=bf]{caption}

\begin{document}

\title[Minimal Filling pair of non orientable surfaces]{Minimal Filling pair of non orientable surfaces}
\author[D. Das]{Debattam Das}
\address{Indian Institute of Technology Kanpur, Kanpur 208016, Uttar Pradesh, India}
\email{debattam123@gmail.com, debattam@iitk.ac.in }
\author[S. Pal]{Souvik Pal }
\address{Indian Institute of Technology Kanpur, Kanpur 208016, Uttar Pradesh, India}
\email{souvikp24@iitk.ac.in}
\author[B. Sanki]{Bidyut Sanki}
\address{Indian Institute of Technology Kanpur, Kanpur 208016, Uttar Pradesh, India}
\email{bidyut@iitk.ac.in}
\makeatletter
\@namedef{subjclassname@2020}{\textup{2020} Mathematics Subject Classification}
\makeatother

\date{\today}
   \begin{abstract}
For $g\ge 3$, let $N_g$ denote the non-orientable surface of genus $g$. In this article, we establish the existence of filling pairs on $N_g$ that intersect minimally by construction using the theory of fat graphs. The mapping class group $\mathrm{Mod}(N_g)$ acts on the set of all such filling pairs. We count $\mathrm{Mod}(N_g)$-orbits of this action  by providing both lower and upper bounds. Furthermore, we show that both bounds grow super-exponentially with $g$ using graph cohomology. Also, we investigate the lengths of minimally intersecting filling pairs on hyperbolic non-orientable surfaces $X$ in moduli space $\mathcal{M}_g$ of $N_g$. We define a function $\mathcal{F}_g:\mathcal{M}_g\to\mathbb{R}_{>0}$, where for $X\in \mathcal{M}_g$, the function $\mathcal{F}_g(X)$ is the shortest total length of a minimally intersecting filling pair on $X$. We determine its minimum $m_g$ and show that the set of minimizers is in bijection with the \(\mathrm{Mod}(N_g)\)-orbits of minimally intersecting filling pairs. We further extend \(\mathcal{F}_g\) to \(\mathcal{Y}_g\), defined by minimizing the length over all filling pairs, and show that \(\mathcal{Y}_g\) attains the same minimum value as \(\mathcal{F}_g\).

\end{abstract}
\maketitle

\section{Introduction}
 Two-dimensional manifold, called \emph{surfaces}, are classified into two types: orientable and non-orientable. A surface is said to be \emph{non-orientable} if it contains an embedded M\"{o}bius strip. Otherwise, it is \emph{orientable}.

By a closed surface, we mean a surface without punctures and boundary. The classification theorem of surfaces (see \cite[Proposition 6.2.7 ]{martelli2016introduction}) says that ,
(i) Every closed connected orientable surface is homeomorphic to the connected sum of $g$ many tori with a sphere, where $g$ is a non-negative integer. Such a surface is denoted by $S_g$ and the integer $g$ is called the genus of the surface. 
(ii) Every closed connected non-orientable surface is homeomorphic to the connected sum of finitely many copies of the real projective plane $\mathbb{RP}^2$.  Such a surface is denoted by $N_g$. where $g $ is called the genus of the surface. Alternatively, the surface $N_g$ may be viewed as a sphere with $g$ \emph{crosscaps} attached. Note that, a crosscap is obtained by removing an open disk from the sphere and identifying antipodal points on the resulting boundary circle.

In this article, by a surface we always mean a closed surface. A \emph{simple closed curve} $\alpha$ on a surface $S$ is the image of an embedding of a circle $\mathbb{S}^1$ into the surface $S$. A simple closed curve is called \emph{trivial} if it bounds a disk or a crosscap. Otherwise, it is called essential. We are interested in essential curves on surfaces.

A collection $\Gamma=\{\gamma_1,\ldots,\gamma_k\}$ of homotopically distinct simple closed curves on a surface $S$ is called a \emph{filling system} if $\gamma_i$'s are in pairwise minimal position (i.e. they do not form a bigon)and every connected component of $S\setminus \bigcup_{i=1}^{k}\gamma_i$ is a topological disk. Equivalently, every essential simple closed curve on $S$ intersects at least one curve in $\Gamma$.

A filling system consisting of exactly two curves is called a \emph{filling pair}. Thus, a pair of simple closed curves $\{\alpha,\beta\}$ forms a filling pair on $S$ if
$
S\setminus (\alpha\cup\beta)
$
is a disjoint union of disks. The filling pair $\{\alpha,\beta\}$ is called \emph{minimally intersecting} if the number $i(\alpha,\beta)$ of intersection points between $\alpha$ and $\beta$ is minimum (for more details see in \cite{MR2850125}). By the Euler characteristic equation, we have,
\[
i(\alpha,\beta)\geq\begin{cases}
2g-1, & S=S_g, \text{ and}\\
g-1, & S=N_g.\\
\end{cases}
\]
Furthermore, equality holds if and only if the complement is a single topological disk. From now on, we call such a filling pair minimal.

The study of filling curves was initiated by Thurston in his unpublished preprint \cite{thurston1986spine}, where he proposed the construction of a spine for the moduli space of Riemann surfaces using filling curves. Recently, Bourque \cite{fortier2024dimension} showed that for every $\varepsilon>0$, there exists a genus $g\geq 2$ such that the dimension of Thurston's spine is at least $(5-\varepsilon)g$.

Filling curves have also played an important role in the study of mapping class groups. In particular, Penner \cite{penner1988construction} used filling multicurves to construct pseudo-Anosov mapping classes.

The combinatorial and topological properties of filling systems on orientable surfaces have been extensively studied. Anderson--Parlier--Pettet \cite{MR2734699}, Sanki \cite{MR3881044}, and Fanoni--Parlier \cite{MR3548116} investigated various aspects of filling systems and their geometric realizations. Aougab--Huang \cite{MR3342680} studied minimal filling pairs on closed orientable surfaces, while Saha--Sanki \cite{MR4926623} examined separating filling pairs on such surfaces. For punctured orientable surfaces, Jeffreys \cite{jeffreys2019minimally} investigated minimally intersecting filling pairs.

In this article, we study the filling system of non-orientable surfaces. To the best of our knowledge, on non-orientable surfaces, filling systems have not yet been investigated in the literature. The only closely related work appears to be \cite{MR4597639}, where the authors describe the possible closures of mapping class group orbits of measured laminations, projective measured laminations, and points in Teichm\"{u}ller space of non-orientable surfaces.
 
 Analogous to the question Aougab--Huang \cite{MR3342680}, here in the context of non-orientable surfaces, we consider the following question.

\begin{question}\label{question 1}
    
 Let $N_g$, $g\in\mathbb{N}$, denote the non-orientable surface of genus $g$. Does there exist a minimal filling pair on $N_g$?
\end{question}

To answer Question~\ref{question 1}, we prove the following theorem.

\begin{theorem}\label{construction theorem}
Let $g>1$, and let $G$ be a $4$-valent fat graph with $g-1$ vertices and two standard cycles. Then there exists a non-orientable surface $N_g^G$, homeomorphic to $N_g$, into which $G$ embeds as a minimal filling pair.
\end{theorem}
We prove the theorem by means of an explicit construction based on $4$-valent fat graphs. More explicitly, we introduce twists on suitable edges in the thickening surface of $G$ to obtain a non-orientable surface with a single boundary component. The genus of the resulting surface is then determined by its Euler characteristic. Furthermore, we prove that for $g=1$, $N_g$ has no minimal filling pair (see Lemma~\ref{g=1} in Section~\ref{section 3}).

Having established the existence of minimal filling pairs on non-orientable surfaces, we proceed to study them up to the action of the mapping class group. Recall that the mapping class group $\mathrm{Mod}(N_g)$ is the group of isotopy classes of self-homeomorphisms of $N_g$.
For each $g\in\mathbb{N}$, let $\mathrm{MF}(N_g)$ denote the set of all minimal filling pairs on the non-orientable surface $N_g$. The group $\mathrm{Mod}(N_g)$ acts on the set $\mathrm{MF}(N_g)$ as follows:
\[
f\cdot\{\alpha,\beta\}=\{f\circ\alpha,f\circ\beta\} ,\;\text{for}\;f\in\mathrm{Mod}(N_g) \;\text{and}\; \{\alpha,\beta\} \in \mathrm{MF}(N_g)
\]
 Let
$\mathcal{N}(g)$
denote the set of $\mathrm{Mod}(N_g)$-orbits of minimal filling pairs. Motivated by Aougab--Huang \cite{MR3342680} and Saha--Sanki \cite{MR4926623}, in the context of minimal filling pairs in non-orientable surfaces, we ask the question below.

\begin{question}
What is the cardinality of $\mathcal{N}(g)$?
\end{question}
To address this question, we prove Theorem~\ref{bounds} which provides an explicit upper bound and a lower bound for $|\mathcal{N}(g)|$.

\begin{theorem}\label{bounds}
Let $N_g$ be the non-orientable surface of genus $g$, and let $\mathcal{N}(g)$ denote the set of $\mathrm{Mod}(N_g)$-orbits of minimally intersecting filling pairs. Then
\[
f(g)\leq |\mathcal{N}(g)|\leq (2^{2(g-1)}-2^{(g-2)})(g-2)!,
\]
where
\[
f(g)\sim C(g-1)^{-9/2}
\left(\frac{g-2}{2e}\right)^g.
\]
In particular, the number of mapping class group orbits of minimal filling pairs grows super-exponentially with $g$.
\end{theorem}
The idea of the proof is as follows: To prove the lower bound, we construct at least $f(g)$ many distinct mapping class group orbits of minimal filling pair of $N_g\;,\;g>1,$ by using the theory of fat graphs and edge-twists. For the upper bound, we introduce vertex-flip operation of fat graphs together with a cohomological argument to obtain the desired estimate.

 Next, we investigate the length of filling pairs on non-orientable surfaces equipped with a hyperbolic metric. For $g\geq 3$, the surface $N_g$ admits a pair of pants decomposition, which implies that $N_g$ can be equipped with a hyperbolic metric. Let $X$ be a hyperbolic non-orientable surface homeomorphic to $N_g$. In the free homotopy class of every closed curve in $X$ has a unique geodesic representative. The length $l_{X}(\alpha)$ of a closed curve $\alpha$ in $X$ is the length of the geodesic representative in the free homotopy class of $\alpha$. The length of a filling pair on $X$ is the sum of the individual lengths of two curves. Now we define a function  

\[
\mathcal{F}_g:\mathcal{M}_g\longrightarrow\mathbb{R},
\qquad
\mathcal{F}_g(\mathcal{X})
=
\min_{(\alpha,\beta)\in\mathcal{N}(g)}
\ell_{\mathcal{X}}(\alpha,\beta),
\]
where $\mathcal{M}_g$ denotes the moduli space of $N_g$. Thus, $\mathcal{F}_g(\mathcal{X})$ measures the length of the shortest minimal filling pair on the hyperbolic surface $\mathcal{X}$.

For orientable surfaces, Aougab--Huang \cite{MR3342680} proved that the analogous length function is proper and a topological Morse function. We show that the similar phenomenon persists in the non-orientable surfaces. Furthermore, we determine the minimum value $m_g$ of $\mathcal{F}_g$.
Now we define the set,
$
\mathcal{B}_g=
\left\{
\mathcal{X}\in\mathcal{M}_g:
\mathcal{F}_g(\mathcal{X})
=
m_g
\right\}.
$
This naturally raises the following question.

\begin{question}
What is the cardinality of $\mathcal{B}_g$?
\end{question}
To answer this question, we show that this set is finite. Moreover, its cardinality is equal to $\mathcal{N}(g)$. So, its cardinality also grows super-exponentially with $g$.

Motivated by the work of Sanki--Vadnere \cite{MR4278333} and Gaster \cite{MR4278331}, we introduce a generalized version of $\mathcal{F}_g$, namely
\[
\mathcal{Y}_g:\mathcal{M}_g\rightarrow\mathbb{R},\qquad
\mathcal{Y}_g(\mathcal{X})
=
\min_{(\alpha,\beta)\in\mathcal{N}'(g)}
\ell_{\mathcal{X}}(\alpha,\beta),
\]
where $\mathcal{N}'(g)$ is the mapping class group orbits of all filling pairs in $N_g$
We prove that $\mathcal{Y}_g$ has the same minimum value as $\mathcal{F}_g$, namely $m_g$.
\section{Preliminaries}

In this section, we recall the background material and fix the notation used in the subsequent sections. We briefly review the theory of fat graphs, non-orientable surfaces, filling pairs and the action of the mapping class groups on the set of filling pairs on non-orientable surfaces, which form the combinatorial framework for our construction. Furthermore, we summarize the necessary facts from hyperbolic geometry and Delaney--Dress symbols that is used to study the length function on the moduli space $\mathcal{M}_g$ of $N_g$.

\subsection{Fat Graphs}
There are several equivalent definitions of graphs, but we give the following definition, which is convenient for defining the fat graphs. 
\begin{definition}\cite[section 2]{MR3881044}
    A \emph{graph} $G$ is a triple $(E,\sim,\sigma_1)$, where
    \begin{enumerate}
        \item $E$ is a finite non-empty set.
        \item $\sim$ is an equivalence relation on $E$.
        \item $\sigma_1:E\rightarrow E$ is a fixed point free involution.
        
    \end{enumerate}

\end{definition}
The set $E$ is called the set of directed edges. The fixed point free involution $\sigma_1$ maps a directed edge to its reverse edge, i.e. 
\[
\sigma_1(\vec{e})=\cev{e}
\]
Hence, the set $E_1=E/\sigma_1$ of orbits of $\sigma_1$ is the set of all undirected edges and the equivalence relation $\sim$ is defined by the following:
\[
\vec{e}_1\sim\vec{e}_2 \iff \vec{e}_1\text{ and }\vec{e}_2\;\; \text{ have the  same initial vertex.}
\]

The set $V=E/\sim$ is the set of all equivalence classes. Each equivalence class is a vertex of the graph. The degree of a vertex is the number of undirected edges in that equivalence class. A graph is called $k$-valent if all the vertices in the graph are of degree $k$.
\begin{definition}
    A \emph{fat graph} is a quadruple $G=(E,\sim,\sigma_1,\sigma_0)$ where 
    \begin{enumerate}
        \item $(E,\sim,\sigma_1)$ is a graph.
        \item $\sigma_0$ is a permutation on $E$ so that each cycle corresponds to a cyclic order on the set of oriented edges going out from a vertex.
    \end{enumerate} 
\end{definition}
Now, we construct a topological surface corresponding to a fat graph $G$. We consider a closed disk corresponding to each vertex and a rectangle to each edge. Next, we identify the two opposite sides of the rectangle with the boundary of the disk according to the order of the edges incident to a vertex. The resulting object is an orientable surface, denoted by $\Sigma(G)$. As a result, each edge of fat graph corresponds to two arcs in the boundary of the surface. 
\begin{figure}[htbp]
    \centering
     
    \tikzset{every picture/.style={line width=0.75pt}} 

\begin{tikzpicture}[,x=0.75pt,y=0.75pt,yscale=-1,xscale=1]

\draw [color={rgb, 255:red, 208; green, 2; blue, 27 }  ,draw opacity=1 ]   (195.54,51.87) -- (194.66,134.41) ;
\draw    (156,94.93) -- (238.6,94.93) ;
\draw    (208.72,111.98) .. controls (227.86,64.02) and (161.74,71.45) .. (185.14,111.65) ;
\draw [shift={(185.88,112.88)}, rotate = 238.31] [color={rgb, 255:red, 0; green, 0; blue, 0 }  ][line width=0.75]    (10.93,-3.29) .. controls (6.95,-1.4) and (3.31,-0.3) .. (0,0) .. controls (3.31,0.3) and (6.95,1.4) .. (10.93,3.29)   ;
\draw   (355.92,94.27) .. controls (355.92,86.47) and (362.11,80.14) .. (369.75,80.14) .. controls (377.39,80.14) and (383.58,86.47) .. (383.58,94.27) .. controls (383.58,102.07) and (377.39,108.39) .. (369.75,108.39) .. controls (362.11,108.39) and (355.92,102.07) .. (355.92,94.27) -- cycle ;
\draw    (382.7,88.22) -- (424,87.76) ;
\draw    (382.7,100.78) -- (424,100.32) ;
\draw    (315.04,89.12) -- (356.34,88.65) ;
\draw    (315.04,101.68) -- (356.34,101.21) ;
\draw [color={rgb, 255:red, 208; green, 2; blue, 27 }  ,draw opacity=1 ]   (365.13,81.48) -- (365.13,46.48) ;
\draw [color={rgb, 255:red, 208; green, 2; blue, 27 }  ,draw opacity=1 ]   (375.67,82.37) -- (375.67,47.38) ;
\draw [color={rgb, 255:red, 208; green, 2; blue, 27 }  ,draw opacity=1 ]   (365.13,141.59) -- (365.13,106.6) ;
\draw [color={rgb, 255:red, 208; green, 2; blue, 27 }  ,draw opacity=1 ]   (375.67,142.48) -- (375.67,107.49) ;
\draw    (374.79,103.01) .. controls (390.29,80.15) and (350.94,77.97) .. (363.42,100.69) ;
\draw [shift={(364.25,102.11)}, rotate = 238.34] [color={rgb, 255:red, 0; green, 0; blue, 0 }  ][line width=0.75]    (10.93,-3.29) .. controls (6.95,-1.4) and (3.31,-0.3) .. (0,0) .. controls (3.31,0.3) and (6.95,1.4) .. (10.93,3.29)   ;

\end{tikzpicture}

    \caption{Local picture of a fat graph and its corresponding surface}
    \label{fig:placeholder}
\end{figure}
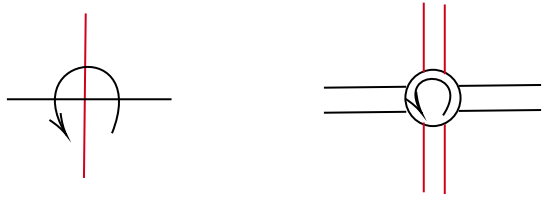

\subsection{Non-orientable surfaces }
A surface is a two-dimensional manifold. By a \emph{closed surface}, we mean a compact connected surface without boundary. A surface is said to be \emph{non-orientable} if it contains an embedded M\"obius strip; otherwise, it is called \emph{orientable}.

The classification theorem for closed surfaces (see \cite[Proposition 6.2.7]{martelli2016introduction}) says that every orientable closed surface is homeomorphic to the connected sum of $g$ many tori, denoted by $S_g$, whereas every non-orientable closed surface is homeomorphic to the connected sum of $g$ copies of the real projective planes, denoted by $N_g$. Equivalently, $N_g$ can be obtained from the $2$-sphere by attaching $g$ crosscaps, where attaching a crosscap means removing an open disk and identifying antipodal points on the resulting boundary circle. The non-negative integer $g$ is called the \emph{genus} of the surface. The Euler characteristics of these surfaces are given by, 
\[
\chi(S_g)=2-2g \text{  and  }
\chi(N_g)=2-g.
\]

\subsection{Filling systems}

A simple closed curve on a surface $S$ is an embedding of a circle $\mathbb{S}^1$ into $S$. A simple closed curve in a closed surface $S$ is called essential if it does not bound a disk or a M\"obius strip.  For two closed curves $\alpha$ and $\beta$ on a surface $S$, the \emph{geometric
intersection number} of the pair $(\alpha,\beta)$ is defined by
\[
i(\alpha,\beta)=\min\{|\alpha'\cap\beta'|:\alpha'\simeq\alpha,\,
\beta'\simeq\beta\},
\]
where $\simeq$ is the free homotopy relation. The curves $\alpha$ and $\beta$ are said to be
in \emph{minimal position} if
$
i(\alpha,\beta)=|\alpha\cap\beta|.
$

\begin{definition}
A collection $\Gamma=\{\gamma_1,\gamma_2,\ldots,\gamma_k\}$ of homotopically distinct essential simple closed curves
on a surface $S$ is called a \emph{filling system} if the curves are pairwise in minimal position and every connected component of $S\setminus\Gamma$ is a topological disk.
\end{definition}
 Let
$
\Gamma=\{\gamma_1,\gamma_2,\ldots,\gamma_k\}
$
be a filling system on a surface \(S\). The set of intersection points of the curves in \(\Gamma\) is defined by
\[
V(\Gamma)
=
\bigcup_{1\le i<j\le k}
(\gamma_i\cap\gamma_j).
\]
A filling system \(\Gamma\) is said to be \emph{minimally intersecting} if \(|V(\Gamma)|\) is minimum among all filling systems on \(S\).

A filling system consisting of two curves is called a \emph{filling pair}. Let \(\Gamma=\{\alpha,\beta\}\) be a filling pair. Since the curves are assumed to be in minimal position, we have
$
|V(\Gamma)|=i(\alpha,\beta),
$
By the Euler characteristic formula,
\[
i(\alpha,\beta)\ge
\begin{cases}
2g-1, & \text{if } S=S_g,\\[2mm]
g-1, & \text{if } S=N_g.
\end{cases}
\]
By Lemma~\ref{minimally intersecting} and \cite[Lemma 2.1]{MR3342680}, equality holds if and only if the complement of the filling pair is a single topological disk. Such filling pairs are called \emph{minimal filling pairs}. For example, in Figure~\ref{filling pair on N2}, $\{\alpha,\beta\}$ is a minimal filling pair of $N_2$

\begin{figure}[htbp]
    \centering

\tikzset{every picture/.style={line width=0.75pt}} 

\begin{tikzpicture}[x=0.75pt,y=0.75pt,yscale=-1,xscale=1]

\draw   (102,90.46) .. controls (102,63.4) and (151.25,41.46) .. (212,41.46) .. controls (272.75,41.46) and (322,63.4) .. (322,90.46) .. controls (322,117.52) and (272.75,139.46) .. (212,139.46) .. controls (151.25,139.46) and (102,117.52) .. (102,90.46) -- cycle ;
\draw   (151,92) .. controls (151,84.27) and (157.27,78) .. (165,78) .. controls (172.73,78) and (179,84.27) .. (179,92) .. controls (179,99.73) and (172.73,106) .. (165,106) .. controls (157.27,106) and (151,99.73) .. (151,92) -- cycle ;
\draw   (148.23,95.69) -- (151.04,89.73) -- (153.69,95.77) ;
\draw   (180.56,88.6) -- (178.52,94.44) -- (175.82,88.87) ;
\draw   (244,92) .. controls (244,84.27) and (250.27,78) .. (258,78) .. controls (265.73,78) and (272,84.27) .. (272,92) .. controls (272,99.73) and (265.73,106) .. (258,106) .. controls (250.27,106) and (244,99.73) .. (244,92) -- cycle ;
\draw   (241.23,95.69) -- (244.04,89.73) -- (246.69,95.77) ;
\draw   (273.56,88.6) -- (271.52,94.44) -- (268.82,88.87) ;
\draw    (165,78) .. controls (189,52.46) and (244,60.46) .. (258,78) ;
\draw    (165,106) .. controls (191,123.46) and (234,123.46) .. (258,106) ;
\draw [color={rgb, 255:red, 208; green, 2; blue, 27 }  ,draw opacity=1 ]   (151,92) .. controls (100,51.46) and (227,53.46) .. (179,92) ;
\draw   (157,59) -- (164,62.23) -- (157,65.46) ;
\draw   (206,59) -- (213,62.23) -- (206,65.46) ;
\draw   (219,122.46) -- (212,119.23) -- (219,116) ;
\draw [color={rgb, 255:red, 208; green, 2; blue, 27 }  ,draw opacity=1 ]   (438.54,47) -- (438.54,128.46) ;
\draw    (438.54,47) -- (520,47) ;
\draw    (438.54,128.46) -- (520,128.46) ;
\draw [color={rgb, 255:red, 208; green, 2; blue, 27 }  ,draw opacity=1 ]   (520,47) -- (520,128.46) ;
\draw   (478,44) -- (485,47.23) -- (478,50.46) ;
\draw   (482.96,131.5) -- (476,128.19) -- (483.04,125.04) ;
\draw   (516.15,91.62) -- (519.62,84.73) -- (522.61,91.84) ;
\draw   (435.25,91.71) -- (438.52,84.73) -- (441.71,91.75) ;

\draw (124,66) node [anchor=north west][inner sep=0.75pt]   [align=left] {$\displaystyle \alpha $};
\draw (234,45) node [anchor=north west][inner sep=0.75pt]   [align=left] {$\displaystyle \beta $};
\draw (476,21) node [anchor=north west][inner sep=0.75pt]   [align=left] {$\displaystyle \beta $};
\draw (474,102) node [anchor=north west][inner sep=0.75pt]   [align=left] {$\displaystyle \beta $};
\draw (419,80) node [anchor=north west][inner sep=0.75pt]   [align=left] {$\displaystyle \alpha $};
\draw (528,77) node [anchor=north west][inner sep=0.75pt]   [align=left] {$\displaystyle \alpha $};
\draw (189,151) node [anchor=north west][inner sep=0.75pt]   [align=left] {(a)};
\draw (469,150) node [anchor=north west][inner sep=0.75pt]   [align=left] {(b)};

\end{tikzpicture}

    \caption{(a) a filling pair $(\alpha,\beta)$ in $N_2$, (b) $N_2-\{\alpha\cup\beta\}$}
    \label{filling pair on N2}
\end{figure}
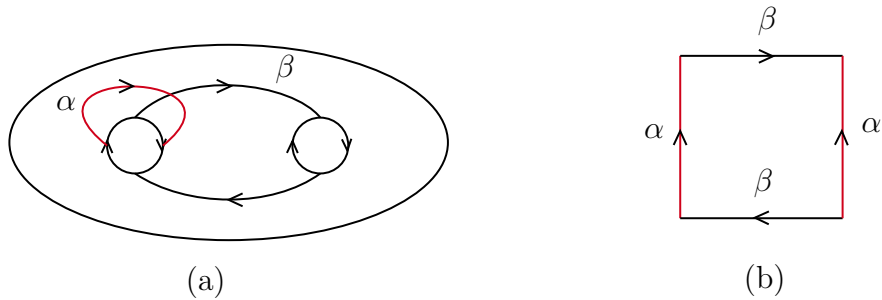
\subsection{Mapping class groups and their action on filling pairs} In this section, we turn our attention to the mapping class groups of non-orientable surfaces $N_g$. The set of all self-homeomorphisms of $N_g$ forms a group under composition of mappings, denoted by $\mathrm{Hom}(N_g).$ We have an equivalence relation $\sim$ on $\mathrm{Hom}(N_g)$: For $f, g\in \mathrm{Hom}(N_g)$, we have $f\sim g$ if $f$ and $g$ are isotopic. Then the set of isotopy classes of $f\in \mathrm{Hom}(N_g)$ denoted by $\mathrm{Mod}(N_g)$, forms a group under the following binary operation: For the equivalence classes $[f]$ and $[g]$ of $f$ and $g$ in $\mathrm{Hom}(N_g)$, we have $[f]\cdot[g]=[fg]$. The group $\mathrm{Mod}(N_g)$ is called the mapping class group of $N_g$, for more details see~\cite{paris2014mapping}. For example, $g=2$, $\mathrm{Mod}(N_2)$ is Klein-$4$ group. This contrasts with the mapping class group of orientable surfaces, where every mapping class group is infinite.

Now, we denote the set of all filling pairs on $N_g$ by $\mathscr A_g$. The mapping class group $ \mathrm{Mod}(N_g)$ acts on $\mathscr A_g$ by the following
\[ [\phi]\cdot (\alpha,\beta)=(\phi\circ\ap, \phi\circ\beta)\] where $[\phi]\in \mathrm{Mod}(N_g)$ and $(\ap,\beta)\in \mathscr A_g.$ We will count the orbit of $\mathrm{Mod}(N_g)$ with respect to this action. 

\subsection{Hyperbolic structures and moduli spaces}

In this subsection, we briefly recall some standard facts concerning
hyperbolic structures on non-orientable surfaces.

The natural analogue of Riemann surfaces in the non-orientable setting is the notion of a Klein surface.
A topological surface equipped with an atlas, whose transition maps are
either holomorphic or anti-holomorphic, is called a \emph{Klein surface}. Such an atlas is called a dianalytic structure on the surface.

 A \emph{non-Euclidean crystallographic} (NEC) group is a discrete subgroup of the isometry group $\mathrm{Isom}(\mathbb{H}^{2})$ of the hyperbolic plane,  possibly containing orientation-reversing isometries. The uniformization theorem for Klein surfaces says that every compact Klein surface with negative Euler characteristic admits a hyperbolic structure.
In particular, if $N_g$ is a closed non-orientable surface of genus
$g\geq 3$, then there exists a torsion-free cocompact NEC group
$\Gamma$ such that the hyperbolic surface $\mathbb{H}^2/\Gamma$ is homeomorphic to $N_g$.

The set of all
complete hyperbolic structures on $N_g$ upto isometries is called the \emph{Moduli space} of $N_g$ and is denoted by $\mathcal{M}_g$. Equivalently, a
point of $\mathcal{M}_g$ represents an isometry class of hyperbolic
surfaces homeomorphic to $N_g$.
Let $\gamma$ be a closed curve in $N_g$. For $X\in\mathcal{M}_g$,
there exists a unique closed geodesic representative $\gamma*$ in the free homotopy class of $\gamma$.
By $l_X(\gamma)$, we mean the length of $\gamma*$ with respect to the hyperbolic structure $X$.We recall the following standard version of the collar lemma, which
guarantees an embedded collar around every simple closed geodesic, with
the collar width determined by its length.
\begin{lemma}[\textbf{Collar lemma}]\label{collar lemma}
        Let $X$ be a complete hyperbolic non-orientable surface. Now consider a simple closed curve $\gamma$ in $X.$ Define
\[
\omega_\gamma
=
\operatorname{arcsinh}
\left(
\frac{1}{\sinh\left(\frac{l_X(\gamma)}{2}\right)}
\right).
\]Then one of the following holds:
        \begin{enumerate}
            \item If $\gamma$ is two-sided, then it has a tubular neighborhood with width of $\omega_\gamma$.
            \item If $\gamma$ is one-sided, then it has a M\"obius band with width of $\omega_\gamma$.
        \end{enumerate}

\end{lemma}
\begin{proof}
First we take the two-sheeted orientable cover of $X$, i.e.,  $\widetilde{X}.$ From, \cite{sayantankhan}, we know that if $\gamma\in X$ is two sided then it can be lifted into two disjoint copies of the same length in $\widetilde{X}$, also if $\gamma\in X$ is one-sided then it can be lifted into a single closed geodesic of the twice of the length of $\gamma.$ Therefore by \cite{MR2850125}, we know that there is a tubular neighborhood of a lift of $\gamma$ in $\widetilde X$ the case of orientable of width $\omega_\gamma.$ That means, the image of this tubular neighborhood in $\widetilde X$ gives a tubular neighborhood of $\gamma$ in $X$ of width $\omega_\gamma$, if $\gamma$ is two-sided, otherwise gives a M\"obius band of width $\omega_\gamma$ which has $\gamma$ as center loop. These give the assertation.
\end{proof}

\subsection{Delaney-Dress symbols}

In this section, we briefly recall the notions of equivariant tilings, chamber systems, and Delaney--Dress symbols. For a detailed exposition, we refer the reader to \cite{huson1993generation}.

Let $\mathcal{T}$ be a tiling of the hyperbolic plane $\mathbb{H}^{2}$ and let $\Gamma$ be a NEC group acting on $\mathbb{H}^{2}$. The pair $(\mathcal{T},\Gamma)$ is called an \emph{equivariant tiling} of $\mathbb{H}^{2}$ if
$
\mathcal{T}=\gamma\mathcal{T}:=\{\gamma(A)\mid A\in\mathcal{T}\}
$
for every $\gamma\in\Gamma$.

Two equivariant tilings $(\mathcal{T},\Gamma)$ and $(\mathcal{T}',\Gamma')$ are said to be \emph{equivariantly equivalent} if there exists a isometry
$
\phi:\mathbb{H}^{2}\longrightarrow \mathbb{H}^{2}
$
such that
$
\phi(\mathcal{T})=\mathcal{T}'
$
and
$
\Gamma'=\phi\Gamma\phi^{-1}.
$

We now describe the chamber system associated with an equivariant tiling. For each vertex of $\mathcal{T}$, choose a distinguished point called a \emph{$0$-center}; for each edge, choose a \emph{$1$-center}; and for each tile, choose a \emph{$2$-center}. Let $A\in\mathcal{T}$ be a tile. Inside $A$, join the $2$-center to each incident $0$-center and $1$-center by pairwise non-intersecting arcs. Repeating this construction for every tile of $\mathcal{T}$ yields a triangulation of $\mathcal{T}$.

The resulting triangles are called \emph{chambers}. Each chamber contains
exactly one $0$-center, one $1$-center, and one $2$-center as its vertices.
The collection of all chambers is denoted by $\mathcal{C}_{\mathcal{T}}$ and
is called the \emph{chamber system} associated with the tiling
$\mathcal{T}$.
The chamber system can be constructed in a way compatible with the action of
$\Gamma$. Let
$
\mathcal{D}:=\mathcal{C}_{\mathcal{T}}/\Gamma
$
be the set of $\Gamma$-orbits of chambers.

For a chamber $T\in\mathcal{C}_{\mathcal{T}}$ and $i\in\{0,1,2\}$, the edge
opposite to the $i$-center is called the \emph{$i$-edge} of $T$. We define
a map
$
\sigma_i:\mathcal{C}_{\mathcal{T}}\longrightarrow
\mathcal{C}_{\mathcal{T}}
$
by letting $\sigma_i(T)$ be the chamber adjacent to $T$ across its
$i$-edge. The compatibility of the chamber system with the action of
$\Gamma$ implies that
$
\sigma_i(\gamma T)=\gamma\sigma_i(T)
$
for every $T\in\mathcal{C}_{\mathcal{T}}$ and every $\gamma\in\Gamma$.
Hence $\sigma_i$ induces a map
\[
\tilde{\sigma}_i:\mathcal{D}\longrightarrow\mathcal{D}.
\]

We define
\[
\mathcal{E}
=
\left\{
(\{D,D'\},i)
\,\middle|\,
D,D'\in\mathcal{D},
\ \tilde{\sigma}_i(D)=D'
\right\}.
\]

For $0\le i<j\le 2$, define a function
\[
m_{ij}:\mathcal{D}\longrightarrow\mathbb{N}
\]
by
\[
m_{ij}(D)
=
\min
\left\{
m\in\mathbb{N}
\;\middle|\;
C(\sigma_i\sigma_j)^m=C
\text{ for all } C\in D
\right\}.
\]

\begin{definition}
A system
\[
(\mathcal{D};m)
:=
\bigl((\mathcal{D},\mathcal{E});
m_{01},m_{02},m_{12}\bigr)
\]
is called a \emph{Delaney--Dress symbol} if for every
$D\in\mathcal{D}$ and every $0\le i<j\le 2$ the following conditions
hold:
    \begin{enumerate}
        \item $m_{ij}(D)=m_{ij}(D\sigma_i)=m_{ij}(D\sigma_j)$
        \item $D(\sigma_i\sigma_j)^{m_{ij}(D)}=D(\sigma_i\sigma_j)^{m_{ij}(D)}=D$
        \item $m_{02}(D)=2$
        \item $m_{01}\geq2$
        \item $m_{12}(D)\geq3$
    \end{enumerate}
\end{definition}
 Two Delaney--Dress symbols $(\mathcal{D};m)$ and
$(\mathcal{D}';m')$ are said to be \emph{isomorphic} if there exists a
bijection
$
\pi:\mathcal{D}\longrightarrow\mathcal{D}'
$
such that
$
\pi(D\tilde{\sigma}_k)
=
\pi(D)\tilde{\sigma}'_k
$
and
$
m'_{ij}(\pi(D))
=
m_{ij}(D)
$
for every $D\in\mathcal{D}$, $0\leq k\leq 2$, and
$0\leq i<j\leq 2$.

We recall the following lemma due to Huson, which connects the equivarient tilings with the Delaney--Dress symbols. 

\begin{lemma}[{\cite[Lemma~1.1]{huson1993generation}}] \label{dd symbol}
Two equivariant tilings $(\mathcal{T},\Gamma)$ and
$(\mathcal{T}',\Gamma')$ are equivariantly equivalent if and only if
their corresponding Delaney--Dress symbols
$(\mathcal{D};m)$ and $(\mathcal{D}';m')$ are isomorphic.
\end{lemma}
\section{Construction of minimal filling pair}\label{section 3}
In this section, we prove Theorem~\ref{construction theorem}. We begin with proving that the geometric intersection number of any filling pair on $N_g$ is at least $g-1$ with equality if and only if it is minimally intersecting. Next, we develop the fat graph construction with suitable edge twists to construct minimal filling pairs on $N_g$ for every $g>1$. We conclude this section by showing that for $g=1$, $N_g$ admits no minimal filling pair.

\begin{lemma}\label{minimally intersecting}
    If $(\alpha,\beta)$ is a filling pair of $N_g$ then $i(\alpha,\beta)\geq g-1$. Furthermore, equality holds if and only if $N_g\setminus\alpha\cup\beta$ is connected.
\end{lemma}
\begin{proof}
    The curves $\alpha$ and $\beta$ fill the surface, implying that $N_g\setminus\alpha\cup\beta$ is a disjoint union of $b$ topological disks. We regard $\alpha\cup \beta$ as the $4$-valent graph embedded in a surface where the intersection points between $\alpha$ and $\beta$ are the vertices and the sub-arcs between the vertices are the edges. The graph $\alpha\cup\beta$ corresponds to a cell decomposition of $N_g$ in which the number of vertices is $i(\alpha,\beta)$, 4-valency of the graph implies that the number of edges is $2i(\alpha,\beta)$, and the number of faces is $b$. Therefore, the Euler characteristics $\chi(N_g)$ of $N_g$ satisfies, 
\begin{align*}
    \chi(N_g)=2-g&=i(\alpha,\beta)-2i(\alpha,\beta)+b\\ \implies i(\alpha,\beta)&=g-2+b
\end{align*} 
    Therefore, \(i(\alpha,\beta)\) attains its minimum value when \(b\) is minimum. Since \(b\geq 1\), we obtain
$
i(\alpha,\beta)\geq g-1.
$
Moreover, equality holds if and only if \(b=1\), that is, if and only if \((\alpha,\beta)\) is a minimal filling pair.

   \end{proof}

Next, we construct minimal filling pairs on $N_g$, for $g\geq 2$. The idea of the construction is the following: By Lemma~\ref{minimally intersecting}, a minimal filling pair on $N_g$ can be considered by a $4$-valent graph with $g-1$ vertices. Motivated by this observation, we begin with a $4$-valent fat graph $G$ having $g-1$ vertices and two standard cycles. The associated orientable surface $\Sigma(G)$ is then modified by inserting edge twists along suitable edges to obtain a non-orientable surface with a single boundary component, denoted by $\Sigma^N(G)$. Finally, capping the boundary by a topological disk yields a closed non-orientable surface $N_g^G$ which is homeomorphic to $N_g$, and the pair of standard cycles corresponds to a minimal filling pair. The idea is illustrated in the example below.
\begin{example}
Consider the $4$-valent fat graph $G$ with two vertices and two standard cycles $\alpha=\alpha_1\cup\alpha_2$ and $\beta=\beta_1\cup\beta_2$ as in Figure~\ref{fat graph and surface} (a). Thickening $G$ produces an orientable surface $\Sigma(G)$ with four boundary components $b_1=\alpha_1\beta_1^{-1},b_2=\alpha_1^{-1}\beta_2^{-1},b_3=\beta_2\alpha_2^{-1},$ and $b_4=\beta_1\alpha_2$. By inserting edge twists along the edges $\alpha_1$, $\alpha_2$, and $\beta_1$, we obtain a non-orientable surface with a single boundary component, as illustrated in Figure~\ref{After inserting the twist}. A straightforward computation of the Euler characteristic shows that this surface has non-orientable genus $3$. Finally, capping the boundary component with a topological disk yields a closed non-orientable surface of genus $3$. The standard cycles $\alpha$ and $\beta$ descend to simple closed curves on this closed surface and form a minimal filling pair on $N_g, g=3$.
\end{example}
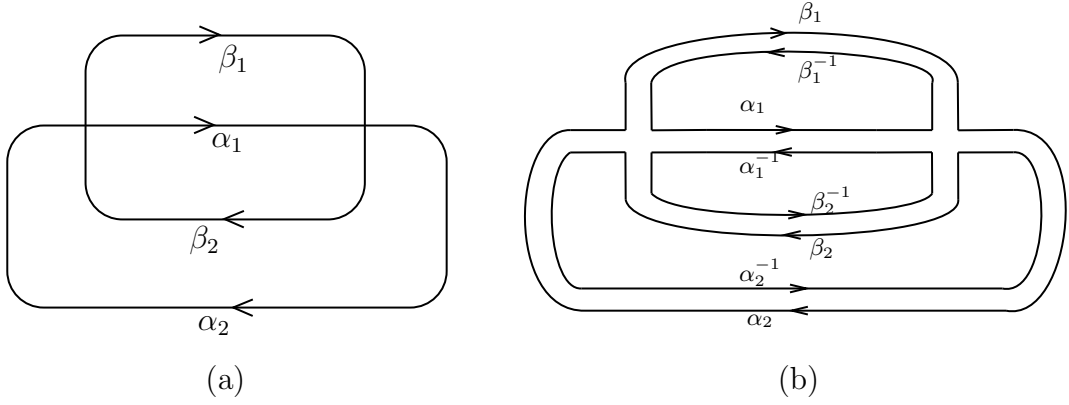
\begin{figure}[htbp]
    \centering

\tikzset{every picture/.style={line width=0.75pt}} 

\tikzset{every picture/.style={line width=0.75pt}} 

\begin{tikzpicture}[x=0.75pt,y=0.75pt,yscale=-1,xscale=1]

\draw   (79,92.82) .. controls (79,82.73) and (87.18,74.55) .. (97.27,74.55) -- (281.17,74.55) .. controls (291.25,74.55) and (299.43,82.73) .. (299.43,92.82) -- (299.43,147.61) .. controls (299.43,157.7) and (291.25,165.88) .. (281.17,165.88) -- (97.27,165.88) .. controls (87.18,165.88) and (79,157.7) .. (79,147.61) -- cycle ;
\draw   (118.27,47.85) .. controls (118.27,37.66) and (126.53,29.4) .. (136.72,29.4) -- (239.93,29.4) .. controls (250.12,29.4) and (258.38,37.66) .. (258.38,47.85) -- (258.38,103.21) .. controls (258.38,113.41) and (250.12,121.67) .. (239.93,121.67) -- (136.72,121.67) .. controls (126.53,121.67) and (118.27,113.41) .. (118.27,103.21) -- cycle ;
\draw    (389.21,53.09) -- (389.21,76.86) ;
\draw    (361.42,76.86) -- (389.21,76.86) ;
\draw    (361.39,88.08) -- (389.11,87.96) ;
\draw    (389.24,111.79) -- (389.11,87.96) ;
\draw    (402.18,108.64) -- (402.2,88.06) ;
\draw    (430,87.99) -- (402.2,88.06) ;
\draw    (429.99,76.77) -- (402.27,76.96) ;
\draw    (402.05,53.13) -- (402.27,76.96) ;
\draw    (543.01,53.09) -- (543.01,76.86) ;
\draw    (515.22,76.86) -- (543.01,76.86) ;
\draw    (515.19,88.08) -- (542.91,87.96) ;
\draw    (543.01,108.64) -- (542.91,87.96) ;
\draw    (556.09,111.83) -- (556.01,88.06) ;
\draw    (583.8,87.99) -- (556.01,88.06) ;
\draw    (583.79,76.77) -- (556.07,76.96) ;
\draw    (555.85,53.13) -- (556.07,76.96) ;
\draw    (429.99,76.77) -- (515.22,76.86) ;
\draw    (430,87.99) -- (515.23,88.08) ;
\draw    (366.98,156.31) -- (578.22,156.31) ;
\draw    (368.83,167.44) -- (580.08,167.44) ;
\draw    (402.05,53.13) .. controls (407.74,29.19) and (550.43,35.55) .. (543.01,53.09) ;
\draw    (389.21,53.09) .. controls (381.8,21.25) and (555.99,18.07) .. (555.85,53.13) ;
\draw    (389.24,111.79) .. controls (394.77,137.25) and (557.84,134.07) .. (556.09,111.83) ;
\draw    (402.18,108.64) .. controls (417.01,124.53) and (535.6,121.36) .. (543.01,108.64) ;
\draw    (578.22,156.31) .. controls (602.31,161.08) and (604.16,86.4) .. (583.8,87.99) ;
\draw    (580.08,167.44) .. controls (613.43,173.79) and (624.55,75.27) .. (583.79,76.77) ;
\draw    (368.83,167.44) .. controls (328.06,168.23) and (331.77,76.07) .. (361.42,76.86) ;
\draw    (366.98,156.31) .. controls (350.3,157.11) and (346.59,98.32) .. (361.39,88.08) ;
\draw   (461.86,25.73) -- (468.99,28.26) -- (461.86,30.78) ;
\draw   (464.53,74.77) -- (472.61,76.72) -- (464.53,78.67) ;
\draw   (472.6,89.78) -- (464.53,87.82) -- (472.61,85.88) ;
\draw   (470.78,117.34) -- (478.85,119.29) -- (470.78,121.23) ;
\draw   (469.03,39.81) -- (460.96,37.86) -- (469.04,35.91) ;
\draw   (477.07,131.41) -- (468.99,129.46) -- (477.07,127.52) ;
\draw   (471.67,154.35) -- (479.74,156.3) -- (471.67,158.25) ;
\draw   (480.64,169.35) -- (472.56,167.4) -- (480.64,165.46) ;
\draw   (172.71,70.15) -- (182.52,74.29) -- (172.71,78.44) ;
\draw   (175.38,24.81) -- (185.2,28.95) -- (175.38,33.1) ;
\draw   (197.66,125.68) -- (187.88,121.43) -- (197.73,117.39) ;
\draw   (202.12,170.09) -- (192.34,165.85) -- (202.19,161.8) ;

\draw (179.88,77.18) node [anchor=north west][inner sep=0.75pt]   [align=left] {$\displaystyle \alpha _{1}$};
\draw (173.09,168.78) node [anchor=north west][inner sep=0.75pt]   [align=left] {$\displaystyle \alpha _{2}$};
\draw (183.52,32.54) node [anchor=north west][inner sep=0.75pt]   [align=left] {$\displaystyle \beta _{1}$};
\draw (168.88,123.29) node [anchor=north west][inner sep=0.75pt]   [align=left] {$\displaystyle \beta _{2}$};
\draw (444.93,60.37) node [anchor=north west][inner sep=0.75pt]  [font=\scriptsize] [align=left] {$\displaystyle \alpha _{1}$};
\draw (444.61,86.17) node [anchor=north west][inner sep=0.75pt]  [font=\scriptsize] [align=left] {$\displaystyle \alpha _{1}^{-1}$};
\draw (474.44,11.33) node [anchor=north west][inner sep=0.75pt]  [font=\scriptsize] [align=left] {$\displaystyle \beta _{1}$};
\draw (474.11,38.05) node [anchor=north west][inner sep=0.75pt]  [font=\scriptsize] [align=left] {$\displaystyle \beta _{1}^{-1}$};
\draw (480.47,103.86) node [anchor=north west][inner sep=0.75pt]  [font=\scriptsize] [align=left] {$\displaystyle \beta _{2}^{-1}$};
\draw (479.9,129.84) node [anchor=north west][inner sep=0.75pt]  [font=\scriptsize] [align=left] {$\displaystyle \beta _{2}$};
\draw (444.61,140.06) node [anchor=north west][inner sep=0.75pt]  [font=\scriptsize] [align=left] {$\displaystyle \alpha _{2}^{-1}$};
\draw (448.5,167.86) node [anchor=north west][inner sep=0.75pt]  [font=\scriptsize] [align=left] {$\displaystyle \alpha _{2}$};
\draw (177,193) node [anchor=north west][inner sep=0.75pt]   [align=left] {(a)};
\draw (464,193) node [anchor=north west][inner sep=0.75pt]   [align=left] {(b)};

\end{tikzpicture}

    \caption{(a) The fat graph $G$, (b) Associated orientable surface $\Sigma(G) $}
    \label{fat graph and surface}
\end{figure}

\begin{figure}[htbp]
    \centering

\tikzset{every picture/.style={line width=0.75pt}} 

\begin{tikzpicture}[x=0.75pt,y=0.75pt,yscale=-1,xscale=1]

\draw    (198.6,79.14) -- (225.59,79.09) ;
\draw    (225.59,56.36) -- (225.59,79.09) ;
\draw    (225.61,113.21) -- (225.58,90.49) ;
\draw    (198.59,90.47) -- (225.58,90.49) ;
\draw    (266,90.57) -- (239.01,90.35) ;
\draw    (238.7,113.08) -- (239.01,90.35) ;
\draw    (239.46,56.23) -- (239.18,78.96) ;
\draw    (266.17,79.24) -- (239.18,78.96) ;
\draw    (361.88,78) -- (388.87,77.95) ;
\draw    (388.87,55.23) -- (388.87,77.95) ;
\draw    (388.89,112.08) -- (388.85,89.35) ;
\draw    (361.86,89.33) -- (388.85,89.35) ;
\draw    (429.27,89.44) -- (402.29,89.22) ;
\draw    (401.98,111.94) -- (402.29,89.22) ;
\draw    (402.74,55.1) -- (402.46,77.82) ;
\draw    (429.44,78.1) -- (402.46,77.82) ;
\draw    (238.7,113.08) .. controls (240.44,125.68) and (379.42,127.96) .. (388.89,112.08) ;
\draw    (225.61,113.21) .. controls (213.45,138.19) and (402.36,143.87) .. (401.98,111.94) ;
\draw    (429.27,89.44) .. controls (446.89,91.59) and (441.5,162.05) .. (426.65,160.91) ;
\draw    (429.44,78.1) .. controls (463.09,77.95) and (461.74,173.41) .. (428,174.55) ;
\draw    (198.59,90.47) .. controls (175.66,99.55) and (186.46,163.19) .. (204,162.09) ;
\draw    (198.6,79.14) .. controls (158.12,77.95) and (160.82,174.55) .. (205.35,175.73) ;
\draw    (225.59,56.36) .. controls (236.39,17.72) and (337.59,48.41) .. (388.87,55.23) ;
\draw    (239.46,56.23) .. controls (247.18,45) and (282.27,43.86) .. (302.51,43.86) ;
\draw    (324,38.96) .. controls (371.23,33.23) and (399.66,40.45) .. (402.74,55.1) ;
\draw    (266.17,79.24) .. controls (309.25,74.55) and (307.91,93.86) .. (361.86,89.33) ;
\draw    (266,90.57) .. controls (280.92,89.32) and (286.32,93.86) .. (307.91,87.05) ;
\draw    (318.7,82.54) .. controls (341.64,76.82) and (355.13,76.82) .. (361.88,78) ;
\draw    (204,162.09) .. controls (274.17,158.64) and (355.13,183.64) .. (428,174.55) ;
\draw    (205.35,175.73) .. controls (232.34,174.55) and (255.28,181.37) .. (303.86,173.41) ;
\draw    (317.35,166.64) .. controls (367.28,159.78) and (349.74,163.19) .. (426.65,160.91) ;
\draw [dotted  ,draw opacity=1 ]   (302.51,43.86) -- (324,38.96) ;
\draw [dotted  ,draw opacity=1 ]   (307.91,87.05) -- (318.7,82.54) ;
\draw [dotted  ,draw opacity=1 ]   (303.86,173.41) -- (317.35,166.64) ;

\draw (307.18,190.53) node [anchor=north west][inner sep=0.75pt]   [align=left] {(c)};

\end{tikzpicture}
    \caption{After inserting the twists}
    \label{After inserting the twist}
\end{figure}
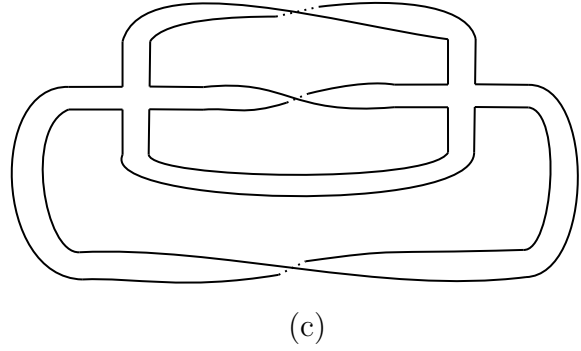

 To construct a minimal filling pair on \(N_g, g\geq 2\), consider a $4$-valent fat graph \(G\) with \(g-1\) vertices and two standard cycles, say \(\alpha\) and \(\beta\). The valency condition on implies that \(G\) has exactly \(2g-2\) edges. We label the edges of \(G\) as follows. Choose an initial edge $\alpha_1$ and a direction on the cycle \(\alpha\). Starting from \(\alpha_1\) and following the direction of \(\alpha\), we label the remaining edges by
$
\alpha_1,\alpha_2,\dots,\alpha_{g-1}.
$
Similarly, we label the edges of the cycle $\beta$ as
$
\beta_1,\beta_2,\dots,\beta_{g-1}.
$
Consider $\Sigma(G)$, the orientable surface associated with $G$. Suppose \(\Sigma(G)\) has \(k\) boundary components, we denoted them by
$
b_1,b_2,\dots,b_k.
$
Each boundary component \(b_i\) is represented by a cyclic word with alphabets in
\[
\mathcal{A}=\{\alpha_1,\ldots,\alpha_{g-1},
\beta_1,\ldots,\beta_{g-1},
\alpha_1^{-1},\ldots,\alpha_{g-1}^{-1},
\beta_1^{-1},\ldots,\beta_{g-1}^{-1}\},
\]
where the exponent indicates the direction in which the corresponding edge is traversed. 

To describe the interaction among the boundary components of
\(\Sigma(G)\), we define an auxiliary graph as follows.  The vertex set of \(\Gamma_G\) is
\[
V(\Gamma_G)=\{b_1,b_2,\ldots,b_k\}.
\]
Two distinct vertices \(b_i\) and \(b_j\) are joined by an edge whenever there exists an edge
$
e\in\{\alpha_1,\ldots,\alpha_{g-1},
\beta_1,\ldots,\beta_{g-1}\}
$
such that \(e\) occurs in the cyclic word representing one of \(b_i\) and
\(e^{-1}\) occurs in the cyclic word representing the other. We denote to
this edge of \(\Gamma_G\) simply by \(e\). We prove the connectedness of $\Gamma_G$ in the following lemma.
\begin{lemma} \label{boundary graph connected}
The graph \(\Gamma_G\) is connected. Consequently, there exists a spanning
tree \(T\) of \(\Gamma_G\). In particular, one can choose \(k-1\) edges
\[
e_1,e_2,\ldots,e_{k-1}\in
\{\alpha_1,\ldots,\alpha_{g-1},
\beta_1,\ldots,\beta_{g-1}\}
\]
such that the corresponding edges of \(\Gamma_G\) connect all the vertices
\(b_1,\ldots,b_k\) and contain no cycle.
\end{lemma}
\begin{proof} The proof is by contradiction, where the main ingredient is the connectedness of $\Sigma(G)$ as $G$ is connected. Assume, \(\Gamma_G\) is not connected. Then the vertex set of \(\Gamma_G\) can be partitioned into two non-empty subsets
$
V(\Gamma_G)=V_1\sqcup V_2,
$
such that there is no edge of \(\Gamma_G\) joining a vertex of \(V_1\) to a vertex of \(V_2\). This induces a partition of the alphabet
\[
\mathcal{A}
=\{\alpha_1,\ldots,\alpha_{g-1},
\beta_1,\ldots,\beta_{g-1},
\alpha_1^{-1},\ldots,\alpha_{g-1}^{-1},
\beta_1^{-1},\ldots,\beta_{g-1}^{-1}\}
\]
into two disjoint subsets
$
\mathcal{A}=\mathcal{A}_1\sqcup\mathcal{A}_2,
$
where \(\mathcal{A}_i\) consists of the letters appearing in the boundary
components belonging to \(V_i\). Since there is no edge of
\(\Gamma_G\) joining a vertex of \(V_1\) to a vertex of \(V_2\), every
edge and its inverse belong to the same subset. Consequently, every edge of \(G\) has both of its occurrences in the same
collection of boundary components. Hence no edge of \(G\) joins a
boundary component in \(V_1\) to one in \(V_2\). It follows that the
thickened surface \(\Sigma(G)\) decomposes as the disjoint union of two
surfaces corresponding to \(V_1\) and \(V_2\). Thus \(\Sigma(G)\) is
disconnected, contradicting the connectedness of \(\Sigma(G)\).
 Therefore, \(\Gamma_G\) is connected. Since every connected graph admits
a spanning tree, the result follows.

\end{proof}

Our next objective is to modify the surface \(\Sigma(G)\) to obtain $\Sigma^N(G)$. This is achieved by introducing twists along suitable edges of \(G\).

Suppose that an edge \(e\) of $G$ and its inverse edge $e^{-1}$  appear in the boundary components $b_1$ and $b_2$ respectively. We write,
\[
b_1=eP,\qquad
b_2=e^{-1}Q,
\]
where \(P\) and \(Q\) are words in the edge alphabet.
To perform a twist along \(e\), we cut the ribbon corresponding to \(e\) along its core and reglue the two resulting half-ribbons after reversing one side. Equivalently, if \(e\) is subdivided into half-edges \(e'\) and \(e''\) as in Figure~\ref{twist between two boundary}(a), then the new edges created by the regluing are
\[
E=e'\bigl((e')^{-1}\bigr)^{-1},
\qquad
E'=\bigl((e'')^{-1}\bigr)^{-1}e''.
\]
This operation merges the boundary components \(b_1\) and \(b_2\) into a single boundary component whose cyclic word is
$
EQ^{-1}E'P.
$
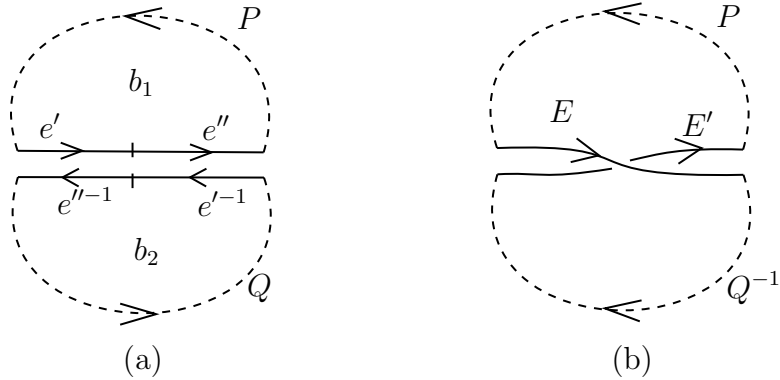
\begin{figure}[htbp]
    \centering

\tikzset{every picture/.style={line width=0.75pt}} 

\begin{tikzpicture}[x=0.75pt,y=0.75pt,yscale=-1,xscale=1]

\draw    (146.63,89.02) -- (270.06,89.8) ;
\draw [dashed  ,draw opacity=1 ]   (146.63,89.02) .. controls (116.9,-1.46) and (297.08,-1.46) .. (270.06,89.8) ;
\draw    (270.22,102.44) -- (146.8,102.2) ;
\draw [dashed  ,draw opacity=1 ]   (270.22,102.44) .. controls (300.44,192.78) and (120.27,193.57) .. (146.8,102.2) ;
\draw   (168.26,84.94) -- (179.07,89) -- (168.26,93.06) ;
\draw   (241.23,106.91) -- (233.13,102.34) -- (241.37,97.98) ;
\draw   (233.12,85.76) -- (243.93,89.82) -- (233.12,93.87) ;
\draw    (204.29,84.94) -- (204.29,92.24) ;
\draw   (177.26,106.91) -- (169.17,102.34) -- (177.41,97.98) ;
\draw    (204.29,100.42) -- (204.29,107.73) ;
\draw   (198.08,165.68) -- (216.1,170.96) -- (198.08,176.24) ;
\draw   (218.23,27.96) -- (200.71,21.45) -- (219.13,17.43) ;
\draw [dashed  ,draw opacity=1 ]   (387.17,87.63) .. controls (357.44,-2.85) and (537.62,-2.85) .. (510.59,88.41) ;
\draw [dashed  ,draw opacity=1 ]   (510.76,101.05) .. controls (540.98,191.39) and (360.8,192.18) .. (387.34,100.81) ;
\draw   (458.77,174.57) -- (441.25,168.06) -- (459.66,164.04) ;
\draw    (387.17,87.63) .. controls (466.45,84.91) and (418.7,102.84) .. (510.76,101.05) ;
\draw    (387.34,100.81) .. controls (408.79,99.58) and (408.79,103.65) .. (444.83,97.95) ;
\draw    (453.84,93.87) .. controls (468.25,90.61) and (478.16,87.36) .. (510.59,88.41) ;
\draw   (427.07,82.76) -- (438.29,91.3) -- (424.11,94.69) ;
\draw   (476.28,82.95) -- (489.82,87.68) -- (478.16,95.5) ;
\draw   (458.77,25.57) -- (441.25,19.06) -- (459.66,15.04) ;

\draw (255.05,15.65) node [anchor=north west][inner sep=0.75pt]   [align=left] {$\displaystyle P$};
\draw (260.3,149.28) node [anchor=north west][inner sep=0.75pt]   [align=left] {$\displaystyle Q$};
\draw (155.9,68.61) node [anchor=north west][inner sep=0.75pt]   [align=left] {$\displaystyle e'$};
\draw (235.34,106.72) node [anchor=north west][inner sep=0.75pt]   [align=left] {$\displaystyle e^{\prime -1}$};
\draw (237.78,69.42) node [anchor=north west][inner sep=0.75pt]   [align=left] {$\displaystyle e''$};
\draw (164.97,105.91) node [anchor=north west][inner sep=0.75pt]   [align=left] {$\displaystyle e''^{-1}$};
\draw (200.75,46.61) node [anchor=north west][inner sep=0.75pt]   [align=left] {$\displaystyle b_{1}$};
\draw (203.45,131.35) node [anchor=north west][inner sep=0.75pt]   [align=left] {$\displaystyle b_{2}$};
\draw (495.59,14.26) node [anchor=north west][inner sep=0.75pt]   [align=left] {$\displaystyle P$};
\draw (500.84,147.89) node [anchor=north west][inner sep=0.75pt]   [align=left] {$\displaystyle Q^{-1}$};
\draw (411.8,62.09) node [anchor=north west][inner sep=0.75pt]   [align=left] {$\displaystyle E$};
\draw (478.22,67.79) node [anchor=north west][inner sep=0.75pt]   [align=left] {$\displaystyle E'$};
\draw (198,185) node [anchor=north west][inner sep=0.75pt]   [align=left] {(a)};
\draw (443,185) node [anchor=north west][inner sep=0.75pt]   [align=left] {(b)};

\end{tikzpicture}
    \caption{(a) Before the twist, (b) After the twist }
    \label{twist between two boundary}
\end{figure}

The above construction describes the effect of a twist along a single edge. We now apply this operation simultaneously to the edges
$
e_1,e_2,\dots,e_{k-1},
$
chosen in Lemma~\ref{boundary graph connected}. Since these edges correspond to the edges of a spanning tree of $\Gamma_G$, each twist merges exactly two distinct boundary components into one. As the spanning tree contains no cycle, every twist decreases the number of boundary components by one. Hence, after performing all the twists, the resulting surface $\Sigma^N(G)$ has exactly one boundary component.
Moreover, each twist inserts a half-twist in the corresponding ribbon, so the resulting surface contains an embedded M\"obius strip. Therefore, $\Sigma^N(G)$ is non-orientable.

The preceding construction applies whenever $\Sigma(G)$ has more than one boundary component. It remains to treat the case where $\Sigma(G)$ already has a single boundary component. In this situation, no twists are needed to merge boundary components. However, to obtain a non-orientable surface, we perform a twist along any edge of $G$. More precisely, let the unique boundary component be
$
b=ePe^{-1}Q,
$
where $P$ and $Q$ are words in the edge alphabet. Subdivide the edge $e$ into two half-edges $e'$ and $e''$, and similarly subdivide $e^{-1}$ into $(e'')^{-1}$ and $(e')^{-1}$. We then reglue the corresponding ribbon after inserting a half-twist, thereby replacing the edge segments by
\[
E=e'\big((e')^{-1}\big)^{-1},
\qquad
E'=\big((e'')^{-1}\big)^{-1}e'' ,
\]
as illustrated in Figure~\ref{fig:internal_twist}. This operation preserves the unique boundary component while introducing an embedded M\"obius strip. Consequently, the resulting surface remains connected, has a single boundary component, and is non-orientable. The new boundary word is
$
EP^{-1}E'Q.
$ and the resulting surface is $\Sigma^N(G)$.
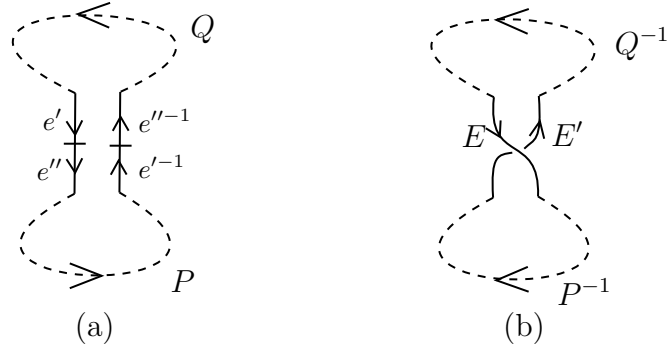
\begin{figure}[htbp]
    \centering

\tikzset{every picture/.style={line width=0.75pt}} 

\begin{tikzpicture}[x=0.75pt,y=0.75pt,yscale=-1,xscale=1]

\draw    (166.77,58.63) -- (166.37,109.08) ;
\draw    (189.3,57.98) -- (188.91,108.43) ;
\draw [dashed  ,draw opacity=1 ]   (166.77,58.63) .. controls (52.09,3.55) and (296.9,9.75) .. (189.3,57.98) ;
\draw [dashed  ,draw opacity=1 ]   (166.37,109.08) .. controls (65.24,161.8) and (286.29,165.77) .. (188.91,108.43) ;
\draw [dashed  ,draw opacity=1 ]   (376.77,60.63) .. controls (262.09,5.55) and (506.9,11.75) .. (399.3,59.98) ;
\draw [dashed  ,draw opacity=1 ]   (376.37,111.08) .. controls (275.24,163.8) and (496.29,167.77) .. (398.91,110.43) ;
\draw    (161,83.96) -- (172,83.96) ;
\draw    (184,84.96) -- (195,84.96) ;
\draw   (170.27,72.26) -- (166.7,79.47) -- (162.33,72.71) ;
\draw   (170.27,91.26) -- (166.7,98.47) -- (162.33,91.71) ;
\draw   (184.56,99.52) -- (188.46,92.48) -- (192.52,99.44) ;
\draw   (185.56,77.52) -- (189.46,70.48) -- (193.52,77.44) ;
\draw   (164,144) -- (180,150.23) -- (164,156.46) ;
\draw   (183.77,25.74) -- (168.01,18.94) -- (184.22,13.29) ;
\draw   (394.2,158.2) -- (378,152.49) -- (393.79,145.74) ;
\draw   (393.7,28.82) -- (378.01,21.86) -- (394.28,16.38) ;
\draw    (376.77,60.63) .. controls (373,95.46) and (400,77.46) .. (398.91,110.43) ;
\draw    (399.3,59.98) .. controls (401,83.46) and (397,81.46) .. (392,86.46) ;
\draw    (376.37,111.08) .. controls (375,89.46) and (379,89.46) .. (386,88.46) ;
\draw   (380.14,72.28) -- (379.72,80.5) -- (372.7,76.21) ;
\draw   (394.66,80.09) -- (399.84,72.72) -- (402,81.46) ;

\draw (149,65) node [anchor=north west][inner sep=0.75pt]  [font=\footnotesize] [align=left] {$\displaystyle e'$};
\draw (146,90) node [anchor=north west][inner sep=0.75pt]  [font=\footnotesize] [align=left] {$\displaystyle e''$};
\draw (197,87.96) node [anchor=north west][inner sep=0.75pt]  [font=\footnotesize] [align=left] {$\displaystyle e^{\prime -1}$};
\draw (197,65) node [anchor=north west][inner sep=0.75pt]  [font=\footnotesize] [align=left] {$\displaystyle e'^{\prime -1}$};
\draw (213,146) node [anchor=north west][inner sep=0.75pt]   [align=left] {$\displaystyle P$};
\draw (223,18) node [anchor=north west][inner sep=0.75pt]   [align=left] {$\displaystyle Q$};
\draw (359,73) node [anchor=north west][inner sep=0.75pt]   [align=left] {$\displaystyle E$};
\draw (404,71) node [anchor=north west][inner sep=0.75pt]   [align=left] {$\displaystyle E'$};
\draw (407,150) node [anchor=north west][inner sep=0.75pt]   [align=left] {$\displaystyle P^{-1}$};
\draw (436,24) node [anchor=north west][inner sep=0.75pt]   [align=left] {$\displaystyle Q^{-1}$};
\draw (165,167) node [anchor=north west][inner sep=0.75pt]   [align=left] {(a)};
\draw (380,167) node [anchor=north west][inner sep=0.75pt]   [align=left] {(b)};

\end{tikzpicture}
    \caption{(a) Before the twist, (b) After the twist}
    \label{fig:internal_twist}
\end{figure}

We have established that every $4$-valent fat graph $G$ with $g-1$ vertices and two standard cycles gives rise to a non-orientable surface $\Sigma^N(G)$ with exactly one boundary component. Observe that $G$ is a deformation retract of $\Sigma^N(G)$. Hence, they have the same Euler characteristic.
Since $G$ has $g-1$ vertices and $2(g-1)$ edges, we obtain
\[
\chi(\Sigma^N(G))
=\chi(G)
=(g-1)-2(g-1)
=-(g-1).
\]
Now, suppose that $\Sigma^N(G)$ has non-orientable genus $g'$. As $\Sigma^N(G)$ has exactly one boundary component, we have
\begin{align*}
\chi(\Sigma^N(G))
&=2-g'-1\\
\implies-(g-1)&=1-g'\\
\implies g'&=g.
\end{align*}
Hence $\Sigma^N(G)$ is a non-orientable surface of genus $g$ with a single boundary component.

Finally, by capping the boundary component with a topological disk, we obtain the closed non-orientable surface $N_g^G$ which is homeomorphic to $N_g$. The fat graph $G$ is naturally embedded in $N_g^G$, and its two standard cycles determine two simple closed curves, denoted by $\alpha$ and $\beta$. Since the complement of $G$ in $N_g^G$ is precisely the capping disk, Hence $(\alpha,\beta)$ is a minimal filling pair on $N_g$. This completes the proof of Theorem~\ref{construction theorem}.
\qed
\begin{lemma}\label{g=1}
    There exists no minimal filling pair of $N_g$, for $g=1$. 
\end{lemma}
\begin{proof}
    Since every essential simple closed curve on $N_1$ is freely homotopic to the unique essential simple closed curve, there do not exist two homotopically distinct essential simple closed curves on $N_1$. Consequently, $N_1$ admits no filling pair, and hence no minimal filling pair.
\end{proof}
\begin{figure}[htbp]
    \centering

\tikzset{every picture/.style={line width=0.75pt}} 

\begin{tikzpicture}[x=0.75pt,y=0.75pt,yscale=-1,xscale=1]

\draw   (262,98.48) .. controls (262,63.95) and (289.99,35.96) .. (324.52,35.96) .. controls (359.05,35.96) and (387.04,63.95) .. (387.04,98.48) .. controls (387.04,133.01) and (359.05,161) .. (324.52,161) .. controls (289.99,161) and (262,133.01) .. (262,98.48) -- cycle ;
\draw   (305.04,95.98) .. controls (305.04,86.05) and (313.09,78) .. (323.02,78) .. controls (332.95,78) and (341,86.05) .. (341,95.98) .. controls (341,105.91) and (332.95,113.96) .. (323.02,113.96) .. controls (313.09,113.96) and (305.04,105.91) .. (305.04,95.98) -- cycle ;
\draw   (300.97,102.01) -- (305.43,92.06) -- (310,101.96) ;
\draw   (345.01,92.09) -- (340.42,101.99) -- (335.98,92.02) ;
\draw [color={rgb, 255:red, 208; green, 2; blue, 27 }  ,draw opacity=1 ]   (323.02,113.96) .. controls (280,152.96) and (270,47.96) .. (323.02,78) ;

\end{tikzpicture}
    \caption{Curves in $N_1$}
    \label{curves in N1}
\end{figure}
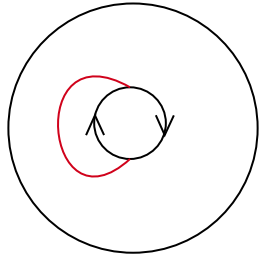



\section{Lower bound of the mapping class group orbit}
In this section, we deduce the lower bound in Theorem~\ref{bounds}. We begin by associating a permutation to every minimal filling pair in Section~\ref{filling permutaion}. Next, we characterize when two such pairs lie in the same mapping class group orbit in terms of these permutations in Proposition~\ref{3.3}. This enables us to bound the size of each mapping class group orbit. Combining this bound with the family of minimal filling pairs obtained from Theorem~\ref{construction theorem}, we derive the desired lower bound in Proposition~\ref{lowerbound}.

\subsection{Filling permutaion}\label{filling permutaion}
Consider $(\alpha,\beta)$, an oriented minimal filling pair on $N_g$, where $g>1$. The union
$
G=\alpha\cup\beta
$
 is naturally an oriented $4$-valent graph with $g-1$ vertices and two standard cycles with
$
E_1(G)=\{\alpha_1,\beta_1,\ldots,\alpha_{g-1},\beta_{g-1}\}
$ (as described in Section~\ref{section 3}). 
Since $(\alpha,\beta)$ fills $N_g$ minimally, we can regard $N_g\setminus G$ as a polygon $P(\alpha,\beta)$. 
Each directed edge of $G$ contributes exactly two directed sides in $P(\alpha,\beta)$. Therefore, the number of sides of $P(\alpha,\beta)$ is a $4g-4$-gon. Traversing the boundary of the polygon $P(\alpha,\beta)$ anticlockwise, we label each side by the corresponding edge with a positive sign if it agrees with the orientation of the edge in $G$; otherwise, we label it with a negative sign. If both sides corresponding to an edge have the same orientation, we distinguish them by writing $\{e^{(1)}, e^{(2)}\}$ or $\{e^{(-1)},e^{(-2)}\},$ according to whether they agree or disagree with the anticlockwise direction of the polygon. If the two occurrences induce opposite directions, we denote them by
$\{e^{(1)},e^{(-1)}\}.$ Therefore, we get the boundary word with letters in the following ordered set,
\begin{align*}
A(g)=\big\{
\alpha_1^{(1)},\beta_1^{(1)},\dots,\alpha_{g-1}^{(1)},\beta_{g-1}^{(1)},
\alpha_1^{(2)},\beta_1^{(2)},\dots,\alpha_{g-1}^{(2)},\beta_{g-1}^{(2)},\\
\alpha_1^{(-1)},\beta_1^{(-1)},\dots,\alpha_{g-1}^{(-1)},\beta_{g-1}^{(-1)},
\alpha_1^{(-2)},\beta_1^{(-2)},\dots,\alpha_{g-1}^{(-2)},\beta_{g-1}^{(-2)} \big\}.
\end{align*}
Let
$
\iota:A(g)\longrightarrow \{1,2,\ldots,8g-8\}
$
 be the order-preserving bijection determined by the above ordering. Replacing each letter of the boundary word by its image under $\iota$, we obtain a $(4g-4)$-cycle in the symmetric group $S_{8g-8}$. We call this cycle the \emph{filling permutation} associated with $(\alpha,\beta)$ and denoted by $\sigma^{(\alpha,\beta)}$. Let $\mathcal{C}(N_g)$ denote the set of all filling permutations.
The filling permutation of the minimal filling pair in Figure~\ref{filling pair on N2} is described in the following example.
\begin{example}
Figure~\ref{filling pair on N2} shows an oriented minimal filling pair $(\alpha,\beta)$ on $N_2$. The associated fat graph has two edges, namely $\alpha$ and $\beta$. Hence,
\[
A(2)=\{\alpha^{(1)},\beta^{(1)},\alpha^{(2)},\beta^{(2)},
\alpha^{(-1)},\beta^{(-1)},\alpha^{(-2)},\beta^{(-2)}\}.
\]
The boundary word is
$
\alpha^{(1)}\beta^{(1)}\alpha^{(-1)}\beta^{(2)}.
$
Thus, the associated filling permutation is
$
(1\;2\;5\;4).
$
\end{example}

\subsection{Characterization of mapping class group orbits}
Having associated a filling permutation to each minimal filling pair, we now describe how these permutations are related when the filling pairs belong to the same mapping class group orbit.
\begin{pro}\label{3.3}
Let
$
\Gamma_1=(\alpha,\beta)$ and 
$\Gamma_2=(\tilde{\alpha},\tilde{\beta})$
be two minimal filling pairs on \(N_g\) lying in the same
\(\mathrm{Mod}(N_g)\)-orbit. 
Then the filling permutations $\sigma^{(\alpha,\beta)}$ and $\sigma^{(\tilde{\alpha},\tilde{\beta})}$ satisfy either
$
\sigma^{(\alpha,\beta)}
=
\phi^{-1}
\sigma^{(\tilde{\alpha},\tilde{\beta})}
\phi,
$
or
$
\sigma^{(\alpha,\beta)}
=
\phi^{-1}
\left(\sigma^{(\tilde{\alpha},\tilde{\beta})}\right)^{-1}
\phi,
$ where
\[
\phi=A^{i}B^{j}C^{m}D^{n}
\prod_{r\in I}E_{\alpha_r}^{p_r}
\prod_{s\in J}E_{\beta_s}^{q_s},
\]
for some 
$
i,j\in\{0,1,\dots,g-2\},
m,n,p^r,q^s\in\{0,1\},
$
and subsets
$
I,J\subseteq\{1,2,\dots,g-1\}.
$

Here the permutations \(A,B,C,D,E_{\alpha_r},E_{\beta_s}\in\Sigma_{8g-8}\) are defined as follows.

\[
\begin{aligned}
A=&
(1\;3\;5\;\dots\;2g-3)
(2g-1\;2g+1\;\dots\;4g-5)\\
&
(4g-3\;4g-1\;\dots\;6g-7)
(6g-5\;6g-3\;\dots\;8g-9),
\end{aligned}
\]

\[
\begin{aligned}
B=&
(2\;4\;6\;\dots\;2g-2)
(2g\;2g+2\;\dots\;4g-4)\\
&
(4g-2\;4g\;\dots\;6g-6)
(6g-4\;6g-2\;\dots\;8g-8),
\end{aligned}
\]

\[
\begin{aligned}
C=&
(1\;4g-3)(3\;4g-1)\cdots(2g-3\;6g-7)\\
&
(2g-1\;6g-5)(2g+1\;6g-3)\cdots(4g-5\;8g-9),
\end{aligned}
\]

\[
\begin{aligned}
D=&
(2\;4g-2)(4\;4g)\cdots(2g-2\;6g-6)\\
&
(2g\;6g-4)(2g+2\;6g-2)\cdots(4g-4\;8g-8),
\end{aligned}
\]

For each \(r\in\{1,2,\dots,g-1\}\),
\[
E_{\alpha_r}
=
(2r-1\;\;2g+2r-3)
(4g+2r-5\;\;6g+2r-7),
\]
and for each \(s\in\{1,2,\dots,g-1\}\),
\[
E_{\beta_s}
=
(2s\;\;2g+2s-2)
(4g+2s-4\;\;6g+2s-6).
\]
\end{pro}

\begin{proof}
Suppose the filling pairs \(\Gamma_1\) and \(\Gamma_2\) lie in the same
\(\mathrm{Mod}(N_g)\)-orbit. Then there exists a homeomorphism
$
f:N_g\to N_g
$
such that
$
f(\Gamma_1)=\Gamma_2.
$

The filling pair determines a polygonal decomposition of \(N_g\) by a
\((4g-4)\)-gon \(P\), whose boundary word gives rise to the associated filling permutation. The homeomorphism \(f\) induces a simplicial automorphism of \(P\), sending vertices to vertices and edges to edges. Since every simplicial automorphism of a polygon preserves the cyclic ordering of its vertices up to reversal, it is induced by an element of the dihedral group \(D_{4g-4}\). There are two cases to be considered depending on whether $f$ is orientation preserving or reversing.

We first assume that \(f\) preserves the orientation of \(P\). Then the induced change in the filling permutation is entirely determined by relabeling the edges, changing their orientations, and interchanging repeated occurrences of the same symbol. Such a relabeling is generated by the following elementary operations.

\medskip

\noindent
\textbf{Case 1.}
Suppose that
$
f(\alpha_n)=\alpha_{n+k_0}
$
for some \(0\leq k_0\leq g-2\), while
$
f(\beta_m)=\beta_m
$
for all \(m,n\in\{1,2,\dots,g-1\}\). Then the \(\alpha\)-indices are cyclically shifted by \(k_0\), and therefore the induced permutation is conjugation by
$
A^{k_0}.
$

\medskip

\noindent
\textbf{Case 2.}
Suppose that
$
f(\alpha_n)=\alpha_n
$
and
$
f(\beta_m)=\beta_{m+k_0}
$
for some \(0\leq k_0\leq g-1\). Then the \(\beta\)-indices are cyclically shifted by \(k_0\), and hence the induced permutation is conjugation by
$B^{k_0}$
.

\medskip

\noindent
\textbf{Case 3.}
Suppose \(f\) preserves the orientation of the \(\beta\)-curves and reverses the orientation of the \(\alpha\)-curves. Then each symbol
$
\alpha_i^{(k)}
$
is replaced by
$
\alpha_i^{(-k)}.
$
Thus the induced permutation is conjugation by \(C\).

\medskip

\noindent
\textbf{Case 4.}
Suppose \(f\) preserves the orientation of the \(\alpha\)-curves and reverses the orientation of the \(\beta\)-curves. Then each symbol
$
\beta_i^{(k)}
$
is replaced by
$
\beta_i^{(-k)},
$
and therefore the induced permutation is conjugation by \(D\).

\medskip

\noindent
\textbf{Case 5.}
Finally, suppose that \(f\) interchanges the two occurrences of certain edge symbols while fixing all the others. More precisely, for some subsets
$
I,J\subseteq\{1,2,\dots,g-1\},
$
the map \(f\) interchanges
\[
\alpha_r^{(1)}\leftrightarrow\alpha_r^{(2)},
\qquad
\alpha_r^{(-1)}\leftrightarrow\alpha_r^{(-2)}
\]
for each \(r\in I\), and
\[
\beta_s^{(1)}\leftrightarrow\beta_s^{(2)},
\qquad
\beta_s^{(-1)}\leftrightarrow\beta_s^{(-2)}
\]
for each \(s\in J\), while leaving all remaining symbols unchanged. The induced permutation is therefore conjugation by
a product of the involutions
$
E_{\alpha_r}
$ and $
E_{\beta_s}.
$
\medskip

Since every orientation-preserving simplicial automorphism of \(P\) is generated by compositions of the above elementary operations, we conclude that
$
\sigma^{(\alpha,\beta)}
=
\phi^{-1}
\sigma^{(\tilde{\alpha},\tilde{\beta})}
\phi,
$
where
\[
\phi
=
A^{i}B^{j}C^{m}D^{n}
\prod_{r\in I}E_{\alpha_r}^{p_r}
\prod_{s\in J}E_{\beta_s}^{q_s}.
\]
Finally, if \(f\) reverses the orientation of \(P\), then the boundary word is read in the opposite direction. Consequently, the associated filling permutation is replaced by its inverse, and hence
$
\sigma^{(\alpha,\beta)}
=
\phi^{-1}
\left(\sigma^{(\tilde{\alpha},\tilde{\beta})}\right)^{-1}
\phi.
$

This completes the proof.

\end{proof}
\begin{remark}\label{remark:orbit-bound}
Let \(\mathcal{O}\) be a \(\mathrm{Mod}(N_g)\)-orbit of minimally intersecting filling pairs. By Proposition~\ref{3.3}, every filling pair in \(\mathcal{O}\) has a filling permutation obtained from any fixed representative by either taking the inverse or conjugating by a permutation of the form
$
A^{i}B^{j}C^{m}D^{n}
\prod_{r\in I}E_{\alpha_r}^{p_r}
\prod_{s\in J}E_{\beta_s}^{q_s}.
$
Since there are \(g-1\) choices for each of \(i\) and \(j\), two choices each for \(m\) and \(n\), and \(2^{g-1}\) choices for each of the subsets \(I\) and \(J\), it follows that a \(\mathrm{Mod}(N_g)\)-orbit contains at most
\[
2\cdot (g-1)^2\cdot 2^2\cdot 2^{g-1}\cdot 2^{g-1}
=
2^{2g+1}(g-1)^2
\]
distinct filling permutations.
\end{remark}

\subsection{A lower bound via fat graphs}
In this subsection, we study the lower bound on the number of the mapping class group orbits of minimally intersecting filling pairs. Namely, we construct a large family of distinct (up to isotopy) minimal filling pairs in the following way. We define an equivalence relation in the set of isomorphism classes of \(4\)-valent fat graphs with \(g-1\) vertices with two standard cycles. Then we associate each equivalence class with a distinct filling pair. Finally, we deduce the lower bound by dividing the number of the equivalence classes by the possible maximum cardinality of the mapping class group orbit.

Let $\mathcal{G}_g$ denote the set of isomorphism classes of $4$-valent fat graphs with $g-1$ vertices and two standard cycles. Let $G_1,G_2\in\mathcal{G}_g$ be two graphs that are isomorphic as abstract graphs and $\varphi: G_1\to G_2$ be an isomorphism. Now, suppose $V(G_1)=\{v_i\mid 1\leq i\leq g-1\}$ is the vertex set of $G_1$ and  $V(G_2)=\{v_i'\mid v_i'=\varphi(v_i), 1\leq i\leq g-1\}$ is the vertex set of $G_2$.   Let 
$
\sigma_0=\sigma_{0,1}\cdots\sigma_{0,g-1}
$
and
$
\sigma_0'=\sigma_{0,1}'\cdots\sigma_{0,g-1}'
$
be the corresponding vertex permutations of $G_1$ and $G_2$, respectively, where $\sigma_{0,i}$ and $\sigma_{0,i}'$ denote the cyclic permutations at $v_i$ and $v_i'$, respectively.

We say that $G_1$ and $G_2$ are \emph{flip-equivalent} either
$
\sigma_{0,i}'=\sigma_{0,i}
$ or $\sigma_{0,i}'=(\sigma_{0,i})^{-1}$,
for every $1\leq i\leq g-1$,
equivalently, $G_2$ is obtained from $G_1$ by reversing the cyclic ordering at some of the vertices and denoted by $G_1\sim_f G_2$. We denote the set of flip-equivalence classes by
$
\mathcal{G}_g^*:=\mathcal{G}_g/\sim_f.
$
In the following lemma, we prove that the elements of ${\mathcal{G}}_g^*$ give distinct minimal filling pairs.

\begin{lemma}\label{graph-injection}
There exists an injective function
$
\Phi:{\mathcal{G}}_g^*\longrightarrow \mathrm{MF}(N_g),
$ where $\mathrm{MF}(N_g)$ is the set of minimal filling pairs on $N_g$.
\end{lemma}
\begin{proof} By Theorem~\ref{construction theorem}, for every
$G\in{\mathcal{G}_g}^*$, there exists a non-orientable surface
$N_g(G)$ homeomorphic to $N_g$ in which the two standard cycles of
$G$ determine a minimal filling pair. Define
$
\Phi(G):=(\alpha_G,\beta_G),
$
where $(\alpha_G,\beta_G)$ denotes this minimal filling pair.

To prove the injectivity of $\Phi$, assume that
$
\Phi(G_1)=\Phi(G_2)
$
for some $G_1,G_2\in{\mathcal{G}_g}^*$. We will prove that
$G_1 =G_2$.
Since $\Phi(G_1)=\Phi(G_2)$, the corresponding minimal
filling pairs are isotopic on $N_g$. Hence there exists an ambient isotopy of $N_g$ sending $\alpha_{G_1}\cup\beta_{G_1}$ onto
$\alpha_{G_2}\cup\beta_{G_2}$ which induces an isomorphism between the graphs $\Phi(G_1)$ and $\Phi(G_2)$. Furthermore, at each vertex, the induced map either preserves or reverses the cyclic ordering of the incident edges, while mapping each standard cycle to the corresponding standard cycle. Consequently, $G_1$ and $G_2$ are flip-equivalent, and hence represent the same element of ${\mathcal{G}}_g^*$.
\end{proof}

In the next step, we estimate the number of flip-equivalence classes. In the following lemma, we derive bounds, both lower and upper, for the cardinality of $\mathcal{G}_g$, from which a lower bound for $|{\mathcal{G}}_g^*|$ will follow immediately.

\begin{lemma}\label{bound of Gg}
    $a_g\leq|\mathcal{G}_g|\leq2^{g-1}(g-2)!$, where $a_g=\frac{2^{g-4}(g-2)!}{(g-1)}.$
\end{lemma}
\begin{proof}
Let $G$ be a member of $\mathcal{G}_g$. First, fix one of the standard cycles, call it \(\alpha\). We choose an initial vertex $v_1$ in $\alpha$, then along the direction of $\alpha$ we label the vertices by $v_1,v_2,\dots,v_{g-2}$ and $v_{g-1}$.
We now count the number of possible choices for the second standard cycle \(\beta\). Fix the starting vertex of \(\beta\) at \(v_1\). For the second vertex $v_j$ along $\beta$, there are $g-2$ choices, and there are exactly two choices of cyclic ordering of the edges incident at $v_j$ as in Figure~\ref{Two possible choice of beta}, provided $\alpha$ and $\beta$ are the standard cycles of $G$.
Hence, there are $2(g-2)$ possible choices for the second vertex of \(\beta\). Similarly, there are $2(g-k$ choices for $k$-th vertex along $\beta$. 

\begin{figure} [htbp]
    \centering

\tikzset{every picture/.style={line width=0.75pt}} 

\begin{tikzpicture}[x=0.75pt,y=0.75pt,yscale=-1,xscale=1]

\draw    (93,143.32) -- (304.35,143.32) ;
\draw    (123.38,152.65) .. controls (114.14,43.53) and (259.44,53.58) .. (252.84,151.21) ;
\draw   (177.54,137.63) -- (189.43,142.99) -- (177.54,148.34) ;
\draw   (178.86,68.72) -- (190.75,74.07) -- (178.86,79.43) ;
\draw    (379.65,101.68) -- (591,101.68) ;
\draw   (464.19,95.99) -- (476.08,101.35) -- (464.19,106.71) ;
\draw    (413.99,111.01) .. controls (417.95,46.4) and (461.55,77.99) .. (461.55,94.5) ;
\draw    (462.87,108.14) .. controls (485.32,169.88) and (539.48,138.29) .. (534.2,90.91) ;
\draw [dotted  ,draw opacity=1 ]   (461.55,94.5) -- (462.87,108.14) ;
\draw   (432.49,67.28) -- (444.37,72.64) -- (432.49,77.99) ;

\draw (132.39,153.6) node [anchor=north west][inner sep=0.75pt]  [font=\footnotesize] [align=left] {$\displaystyle v_{1}$};
\draw (233.79,155.04) node [anchor=north west][inner sep=0.75pt]  [font=\footnotesize] [align=left] {$\displaystyle v_{j}$};
\draw (181.95,123.45) node [anchor=north west][inner sep=0.75pt]  [font=\footnotesize] [align=left] {$\displaystyle \alpha $};
\draw (179.31,50.23) node [anchor=north west][inner sep=0.75pt]  [font=\footnotesize] [align=left] {$\displaystyle \beta $};
\draw (468.6,81.82) node [anchor=north west][inner sep=0.75pt]  [font=\footnotesize] [align=left] {$\displaystyle \alpha $};
\draw (396.58,109.09) node [anchor=north west][inner sep=0.75pt]  [font=\footnotesize] [align=left] {$\displaystyle v_{1}$};
\draw (540.25,109.09) node [anchor=north west][inner sep=0.75pt]  [font=\footnotesize] [align=left] {$\displaystyle v_{j}$};
\draw (438.21,45.92) node [anchor=north west][inner sep=0.75pt]  [font=\footnotesize] [align=left] {$\displaystyle \beta $};

\end{tikzpicture}
    \caption{Two possible choice of $\beta$}
    \label{Two possible choice of beta}
\end{figure}
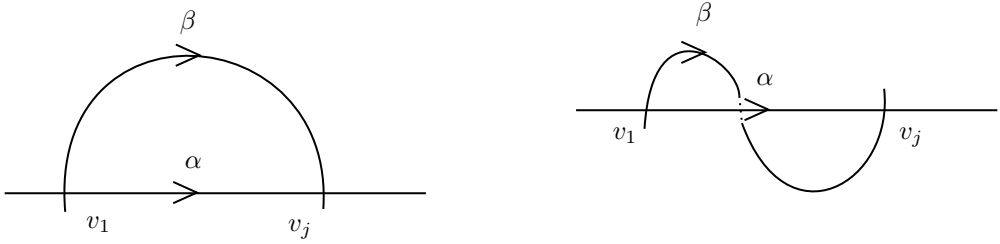
Consequently, the total number of all possibilities is $2^{g-1}(g-2)!.$
It follows that 
\[\vert \mathcal G_g\vert\le 2^{g-1}(g-2)!\]

Now, we deduce the lower bound. First, we change the initial vertex from $v_1$ to $v_i$ for some $1\le i\le g-1.$ The graphs, which we are getting from the above choices, also can be found in the case of taking $v_i$ as first vertex. So there are $ (g-1)$ repetition of the graphs from the above choices since there are $g-1$ vertices and we can start any one of them. Also, we can swap $\alpha$ and $\beta$, and change the direction of $\alpha$ and $\beta$. This operations may also repeat in the total possibilities. Therefore, the total number of distinct isomorphism classes \(4\)-valent fat graphs with \(g-1\) vertices and two standard cycles is at least $\frac{2^{g-1}(g-2)!}{2\cdot 2\cdot2\cdot(g-1)},$ i.e.,
\[\lvert\mathcal{G}_g \rvert\ge \frac{2^{g-1}(g-2)!}{8(g-1)}=\frac{2^{g-4}(g-2)!}{(g-1)}. \]
\end{proof}
\begin{remark}
By Lemma~\ref{bound of Gg}, we have
$
|\mathcal{G}_g|\ge a_g.
$
Each flip-equivalence class contains at most $2^{g-1}$ elements, since at each of the $g-1$ vertices there are at most two choices for the cyclic ordering (the given ordering or its reverse). Therefore,
$
|{\mathcal{G}}_g^*|
\ge
\frac{|\mathcal{G}_g|}{2^{g-1}}
\ge
\frac{a_g}{2^{g-1}}.
$
\end{remark}
Finally, we combine the lower bound of $|{\mathcal{G}}_g^*|$ with the injective map in Lemma~\ref{graph-injection} and the characterization of the mapping class group orbits given in Proposition~\ref{3.3} to deduce the desired lower bound on the number of mapping class group orbits. In particular, we prove the following proposition.
\begin{pro}
    \label{lowerbound}
Let $N(g)$ denote the number of $\mathrm{Mod}(N_g)$-orbits of minimal filling pair on $N_g$, for $g>2$. Then
\[
N(g)\gtrsim f(g),
\]
where
$
f(g)= \frac{e^2\sqrt{2\pi}}{32}(g-1)^{-9/2}
\left(\frac{g-2}{2e}\right)^g.
$
In particular, the number of mapping class group orbits of minimal filling pairs grows super-exponentially with $g$.
\end{pro}
\begin{proof} 
By Lemma~\ref{graph-injection}, there is an injective map
$
\Phi:{\mathcal{G}}_g^*\longrightarrow \mathrm{MF}(N_g).
$
Hence the number of minimal filling pairs on $N_g$ is at least
$|{\mathcal{G}}_g^*|$. By Remark~\ref{bound of Gg}, we have
$
|{\mathcal{G}}_g^*|
\ge
\frac{(g-2)!}{32(g-1)}.
$
By Remark~\ref{remark:orbit-bound}, every
$\mathrm{Mod}(N_g)$-orbit contains at most
$
2^{2g+1}(g-1)^2
$
distinct filling permutations, and hence at most this many minimal
filling pairs. Therefore,
\[
N(g)
\ge
\frac{|{\mathcal{G}}_g^*|}
{2^{2g+1}(g-1)^2}
\ge
\frac{(g-2)!}
{2^{g+5}(g-1)^3}
=:F(g).
\]

Applying Stirling's formula,
$
(g-2)!
\sim
\sqrt{2\pi(g-2)}
\left(\frac{g-2}{e}\right)^{g-2},
$
we obtain
\[
F(g)
=
\frac{e^{2}\sqrt{2\pi}}{32}
\cdot
\frac{(g-2)^{\,g-\frac32}}
{(2e)^g(g-1)^3}
\left(1+O(g^{-1})\right)
\sim
f(g),
\]
tumiHence, we obtain $N(g)\gtrsim f(g).$ 
\end{proof}

\begin{remark}
  In particular, unlike orientable case (see \cite{MR3342680}), the number of mapping class group orbits of minimal filling pairs for $N_g$ grows super-exponentially with $g$. To see this
  \[
\frac{\log f(g)}{g}
=
\log\!\left(\frac{g-2}{2e}\right)
-
\frac{9}{2}\frac{\log(g-1)}{g}
+
\frac{\log C}{g}
\longrightarrow\infty,
\text{  as  } g\to\infty.
\]
 
\end{remark}

\section{Upper bound of the mapping class group orbit}
In this section, our goal is to prove the upper bound in \thmref{bounds}. Here, we first introduce twist-labeled fat graphs and define an operation on them, which we call the \textit{vertex-flip operation}. We then discuss the \(\mathbb{Z}/2\z\)-cohomological interpretation of twist functions and show that the corresponding cohomology classes are invariant under vertex-flip operations. Then we show that the minimal filling pairs can be seen as Twist-labeled fat graphs.  Finally, by counting the twist labeled fat graphs we obtain the desired upper bound in \thmref{bounds}.

\subsection{Twist-labeled fat graph and Vertex flip operation}

In this subsection, we introduce the twist functions and twist-labeled fat graphs which associates with non orientable surface. It helps to construct flip-equivalent fat graphs.
\begin{definition}
A \emph{twist function} on $4$-valent fat graph $G$ is defined by a map
$
\tau:E_1(G)\longrightarrow\z/2\z.
$
The pair of a fat graph $G$ and a twist function $\tau$, i.e., $(G,\tau)$ is called a \emph{twist-labeled fat graph}. 
\end{definition}
 Given a twist-labeled fat graph $(G,\tau)$, we construct a surface $\Sigma(G,\tau)$ from $\Sigma(G)$, associated to $G$ as follows. For each edge $e\in E_1(G)$, we perform a twist along $e$ precisely when $\tau(e)=1$. In other words, the value $\tau(e)$ records whether or not the corresponding edge is twisted in the construction of $\Sigma(G,\tau)$. Thus, the twist function $\tau$ completely determines the twisting data and, consequently, the associated surface $\Sigma(G,\tau)$.

We now define a local operation on twist-labeled fat graphs, called
the \emph{vertex flip operation}, which originally appeared in the context
of ribbon graph duality in  \cite{mulase2003duality}.

\begin{definition}[\textit{Vertex flip}]
Let $(G,\tau)$ be a twist-labeled fat graph and let $v$ be a vertex of
$G$. The \emph{vertex flip} in $v$ is the operation $F_v$ that sends $(G,\tau)$ to the twist-labeled fat graph $(F_v(G),\tau_{v})$
where
\begin{enumerate}

    \item $F_v(G)$ is the fat graph obtained from $G$ by reversing the cyclic ordering of the edges incident to $v$; and
    \item $\tau_v$ is the twist function on $F_v(G)$ defined by
    \[
    \tau_v(e)=
    \begin{cases}
    1-\tau(e), & \text{if } e \text{ is incident to } v,\\
    \tau(e), & \text{otherwise}.
    \end{cases}
    \]
\end{enumerate}
\end{definition}

\begin{figure} [htbp]
    \centering

\tikzset{every picture/.style={line width=0.75pt}} 

\begin{tikzpicture}[x=0.75pt,y=0.75pt,yscale=-1,xscale=1]

\draw    (157.03,30.99) -- (157.03,90.56) ;
\draw    (99.71,90.56) -- (157.03,90.56) ;
\draw    (234.07,90.75) -- (176.75,90.36) ;
\draw    (177.12,30.8) -- (176.75,90.36) ;
\draw    (176.77,170.21) -- (176.93,110.65) ;
\draw    (234.26,110.82) -- (176.93,110.65) ;
\draw    (99.4,110.54) -- (156.73,110.06) ;
\draw    (157.19,169.62) -- (156.73,110.06) ;
\draw    (416.46,61.54) -- (416.46,89.86) ;
\draw    (388.28,90.77) -- (416.46,89.86) ;
\draw    (460.45,90.02) -- (431.33,89.7) ;
\draw    (430.71,62.29) -- (431.33,89.7) ;
\draw    (432.31,130.04) -- (432.23,101.72) ;
\draw    (458.73,101.74) -- (432.23,101.72) ;
\draw    (388.28,102.65) -- (417.36,101.91) ;
\draw    (418.06,129.32) -- (417.36,101.91) ;
\draw    (345.07,103.56) .. controls (368.55,106.3) and (366.67,89.86) .. (388.28,90.77) ;
\draw    (345.07,91.69) .. controls (368.55,94.43) and (366.67,101.74) .. (388.28,102.65) ;
\draw    (458.73,101.74) .. controls (480.33,101.74) and (479.4,88.03) .. (501,88.95) ;
\draw    (459.67,89.86) .. controls (483.15,92.6) and (479.4,100.82) .. (501,101.74) ;
\draw    (416.46,61.54) .. controls (417.4,24.08) and (466.24,13.12) .. (465.3,87.12) ;
\draw    (430.71,62.29) .. controls (431.49,46.01) and (434.31,36.87) .. (440.88,35.04) ;
\draw    (446.52,31.39) .. controls (476.42,31.55) and (477.52,68.85) .. (477.52,83.46) ;
\draw    (418.06,129.32) .. controls (420.22,185.79) and (368.55,192.18) .. (367.61,111.78) ;
\draw    (432.31,130.04) .. controls (433.37,163.4) and (422.09,172.54) .. (406.13,170.71) ;
\draw    (381.7,113.16) .. controls (382.64,141.48) and (389.22,159.75) .. (399.55,167.06) ;
\draw    (252.96,100.53) -- (325.49,101.26) ;
\draw [shift={(327.49,101.28)}, rotate = 180.57] [color={rgb, 255:red, 0; green, 0; blue, 0 }  ][line width=0.75]    (10.93,-3.29) .. controls (6.95,-1.4) and (3.31,-0.3) .. (0,0) .. controls (3.31,0.3) and (6.95,1.4) .. (10.93,3.29)   ;

\end{tikzpicture}
    \caption{Vertex flip operation}
    \label{vertex flip opertation}
\end{figure}
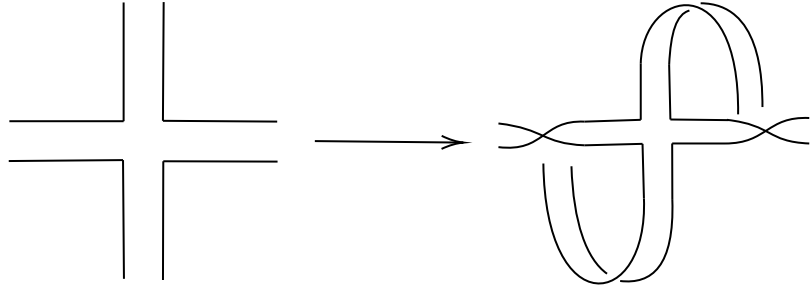
 We say that $(G_1,\tau_1)$ and $(G_2,\tau_2)$ are \emph{flip-equivalent} if one can be obtained from the another by a finite sequence of vertex flips, and write $ (G_1,\tau_1)\sim_f(G_2,\tau_2).$ It is immediate from the definition that, if $ (G_1,\tau_1)\sim_f(G_2,\tau_2)$, then $G_1\sim_f G_2$. 

\subsection{Cohomological interpretation of twist functions}

A $4$-valent graph $G$ can be considered as a CW-complex whose $0$-cells are the vertices and whose
$1$-cells are the edges. The cellular cochain complex with coefficients
in $\mathbb{Z}/2\mathbb{Z}$ is
\[
0 \longrightarrow
C^0(G;\mathbb{Z}/2\mathbb{Z})
\xrightarrow{\;\delta\;}
C^1(G;\mathbb{Z}/2\mathbb{Z})
\longrightarrow 0,
\]
where
$
C^k(G;\mathbb{Z}/2\mathbb{Z})
=
\mathrm{Maps}
\bigl(
\{\text{$k$-cells of $G$}\},
\mathbb{Z}/2\mathbb{Z}
\bigr),
k=0,1.
$
In particular,
$
C^0(G;\mathbb{Z}/2\mathbb{Z})
\cong
(\mathbb{Z}/2\mathbb{Z})^{V(G)}
\text{ and }
C^1(G;\mathbb{Z}/2\mathbb{Z})
\cong
(\mathbb{Z}/2\mathbb{Z})^{E_1(G)}.
$ Now we define the coboundary map $\delta$.
\begin{definition}[The coboundary map \texorpdfstring{$\delta$}{delta}]\label{def:delta}
For $f \in C^0(G;\,\mathbb{Z}/2\mathbb{Z})$ and an edge $e$ with
tail $\partial_0 e$ and head $\partial_1 e$, define
\begin{align*}
    (\delta f)(e)
  &= f(\partial_1 e) - f(\partial_0 e)
 \mod 2\\ &=f(\partial_1 e) + f(\partial_0 e)\mod 2.
\end{align*}
  
\end{definition}

In words, $(\delta f)(e) = 1$ if and only if $f$ takes \emph{different}
values on the two endpoints of $e$, otherwise $(\delta f)(e) = 0.$ Now we can see the vertex flip operation will not change the cohomology.

A twist function
$
\tau:E_1(G)\longrightarrow \z/2\z
$
may naturally be regarded as a cellular $1$-cochain with coefficients in
$\mathbb{Z}/2\mathbb{Z}$. Thus the set of all twist functions on $G$ is
precisely the cochain group
$
C^1(G;\mathbb{Z}/2\mathbb{Z}).
$
\(F_v\) is a function from the set of all twist-labeled fat graphs to itself, defined by
\[
(G,\tau)\longmapsto \bigl(F_v(G),\tau_v\bigr).
\]
Since the vertex-flip operation does not change the edge set, the edges of \(G\) and \(F_v(G)\) are the same. Consequently, we have a natural identification
\[
C^1(G;\mathbb{Z}/2\mathbb{Z})
\cong
C^1(F_v(G);\mathbb{Z}/2\mathbb{Z}).
\]
Under this identification, \(F_v\) induces a map
\[
F_v^*\colon C^1(G;\mathbb{Z}/2\mathbb{Z})
\longrightarrow C^1(G;\mathbb{Z}/2\mathbb{Z}),
\]
given by
\[
\tau\longmapsto F_v^*(\tau)=\tau_v.
\]
Thus, we may regard the vertex-flip operation as inducing an action on the space of \(1\)-cochains of \(G\).

\begin{lemma}\label{flip-cobdy}
A vertex-flip at vertex $v$ corresponds exactly to replacing $\tau$ by
$\tau + \delta f_v$, where $f_v \in C^0(G;\,\mathbb{Z}/2\mathbb{Z})$
is the characteristic function of $v$:
\[
  f_v(u) =
  \begin{cases}
    1 & u = v,\\
    0 & u \neq v.
  \end{cases}
\]
More generally, a sequence of vertex-flips at a subset $S \subseteq V(G)$
replaces $\tau$ by $\tau + \delta f_S$, where $f_S$ is the indicator
of $S$.
\end{lemma}

\begin{proof}
For any edge $e\in E$,
by Definition~\ref{def:delta}, $(\delta f_v)(e) = 1$ if and only if
exactly one endpoint of $e$ is $v$, i.e.\ $e$ is incident to $v$.
Adding $\delta f_v$ to $\tau$ makes a change in the value of $\tau(e)$ 
for the edges incident to $v$, which is exactly the effect of a vertex-flip
at $v$.  This proves our first assertation. 
The general statement follows by applying the vertex flip operations finitely many times. 
\end{proof}

Now, we say that two twist functions $\tau_1$ and $\tau_2$ on a $4$-valent fat graph $G$ are \emph{flip-equivalent} if there exist fat graphs $G_1$ and $G_2$ with $G_1\sim_f G_2$ such that the twist-labeled fat graphs $(G_1,\tau_1)$ and $(G_2,\tau_2)$ are flip-equivalent.

Note that a vertex flip changes the twist function by a coboundary. This leads to the following characterization of flip-equivalence in terms of cohomology classes. That means, if $B^1(G;\mathbb{Z}/2\mathbb{Z})=Im(\delta)$, $F_v$ preserves $B^1(G;\mathbb{Z}/2\mathbb{Z})$ for all vertex $v.$

\begin{lemma}\label{flip-cohomology}
Let $G$ be a $4$-valent fat graph and let
$\tau_1,\tau_2\in C^1(G;\mathbb Z/2\mathbb Z)$ be two twist cochains.
Then $\tau_1$ and $\tau_2$ are flip-equivalent if and only if they
represent the same cohomology class in
$H^1(G;\mathbb Z/2\mathbb Z)$.

\end{lemma}

\begin{proof}
Suppose that $\tau_1$ and $\tau_2$ are flip-equivalent. Then there exists a sequence of vertex-flips transforming
$(G,\tau_1)$ into a twist-labeled fat graph equivalent to
$(G',\tau_2)$, where $G$ and $G'$ are flip-equivalent. So, they are same as graphs. 
By Lemma~\ref{flip-cobdy}, A sequence of vertex-flips replaces $\tau_1$ by
$
\tau_1+\delta f
$
for some $f\in C^0(G;\mathbb Z/2\mathbb Z)$.
Since vertex-flips do not alter the underlying edge set of $G$, we obtain
$
\tau_2=\tau_1+\delta f.
$
Therefore
$
[\tau_1]=[\tau_2]
$
in $H^1(G;\mathbb Z/2\mathbb Z)$.

Conversely, suppose that
$
[\tau_1]=[\tau_2].
$
Then
$
\tau_2-\tau_1\in B^1(G;\mathbb Z/2\mathbb Z),
$
so there exists
$f\in C^0(G;\mathbb Z/2\mathbb Z)$ such that
$
\tau_2-\tau_1=\delta f.
$
By Lemma~\ref{flip-cobdy}, the cochain $\delta f$ is realized by a
sequence of vertex-flips.
Thus $\tau_1$ and $\tau_2$ are flip-equivalent.
\end{proof}

\subsection{Minimal filling pair as twist-labeled fat graph} \label{Minimal filling pair as twist-labeled fat graph}
We now associate to each minimal filling pair on $N_g$ a twist-labeled fat graph, thereby obtaining a combinatorial description of minimal filling pairs. Define
\[
\widehat{\mathcal{G}_g}=
\left\{
(G,\tau)
\;\middle|\;
\begin{array}{l}
G\in \mathcal{G}_g,
\tau\in C^1(G,\z/2\z) \\
\left[\tau\right]\neq0,\text{ and } \\
\Sigma(G,\tau)\text{ has a single boundary component}
\end{array}
\right\}.
\]
By the construction of Section~\ref{section 3}, every element of $\mathcal{G}_g^{\mathrm{tw}}$ determines a minimal filling pair on $N_g$. Conversely, let $(\alpha,\beta)$ be a minimal filling pair on $N_g$.
Then the union
$
G=\alpha\cup\beta
$
is a $4$-valent graph with $g-1$ vertices. Since every surface is locally orientable, the local orientation at each vertex induces a cyclic order on the incident edges. Hence, $G$ becomes a $4$-valent fat graph whose two standard cycles are precisely $\alpha$ and $\beta$. Define the twist function
$
\tau:E_1(G)\longrightarrow \z/2\z
$
by setting $\tau(e)=1$ if the corresponding segment of $\alpha\cup\beta$
passes through an odd number of cross-caps, and $\tau(e)=0$ otherwise.
Thus, $(\alpha,\beta)$ determines a twist-labeled fat graph $(G,\tau)$.Since $(\alpha,\beta)$ fills $N_g$ minimally, $\Sigma(G,\tau)$ has a single
boundary component. Moreover, $\Sigma(G,\tau)$ is non-orientable and, hence
$[\tau]\neq0$. Therefore, each minimal filling pair determines an element
of $\widehat{\mathcal{G}_g}$. However, this correspondence is not bijective:
an isotopy of $(\alpha,\beta)$ may change the induced cyclic orderings at
the vertices, and hence may produce a different fat-graph representative.

\begin{definition}
The twist-labeled fat graphs $(G_1,\tau_1)$ and $(G_2,\tau_2)$
are mapping class group equivalent if there exists a homeomorphism
$
f:\Sigma(G_1,\tau_1)\to \Sigma(G_2,\tau_2)
$
such that $f$ sends each standard cycle of $G_1$ to
standard cycle of $G_2$. It is denoted by $(G_1,\tau_1)\sim_M(G_2,\tau_2)$

\end{definition}
The above notion of mapping class group equivalence is consistent with the
usual action of the mapping class group on minimal filling pairs. Indeed, if
$(G_1,\tau_1)$ and $(G_2,\tau_2)$ are mapping class group equivalent, then
the associated minimal filling pairs lie in the same mapping class group
orbit, since the homeomorphism carries the standard cycles of
$(G_1,\tau_1)$ to those of $(G_2,\tau_2)$. 
Conversely, suppose that two minimal filling pairs
$(\alpha_1,\beta_1)$ and $(\alpha_2,\beta_2)$ lie in the same mapping
class group orbit. The twist-labeled fat graphs
$(G_1,\tau_1)$ and $(G_2,\tau_2)$ are associated with
$(\alpha_1,\beta_1)$ and $(\alpha_2,\beta_2)$, respectively. Then there
exists a homeomorphism
$
f:N_g\longrightarrow N_g
$
sending $(\alpha_1,\beta_1)$ to $(\alpha_2,\beta_2)$. Restricting $f$ to
the regular neighborhoods of the filling pairs yields a homeomorphism
$\Sigma(G_1,\tau_1)$ to $\Sigma(G_2,\tau_2),$
which sends the standard cycles of $G_1$ to those of $G_2$. Hence
$(G_1,\tau_1)$ and $(G_2,\tau_2)$ are mapping class group equivalent.

Let $\widehat{\mathcal{G}_g}/\sim_M$
denote the set of mapping class group equivalence classes of
$\widehat{\mathcal{G}_g}$. The preceding discussion gives a bijection between $\widehat{\mathcal{G}_g}/\sim_M$ and the mapping class group orbits of minimal filling pairs on $N_g.$
Therefore,
$
|\widehat{\mathcal{G}_g}/\sim_M| = N(g).
$
\begin{lemma}\label{flip eq vs mcg eq}
Let $(G_1,\tau_1),(G_2,\tau_2)\in\widehat{\mathcal{G}_g}$ be flip-equivalent.  Then they are mapping class group equivalent.
\end{lemma}
\begin{proof}
    It is enough to prove the statement when $(G_2,\tau_2)$ is obtained from $(G_1,\tau_1)$ by a single vertex flip. 
Let $v$ be the vertex where the flip is performed. As illustrated in Figure~\ref{vertex flip opertation},
the vertex flip is realized by a homeomorphism between suitable
neighborhoods of $v$ in $\Sigma(G_1,\tau_1)$ and $\Sigma(G_2,\tau_2)$,
which preserves the two standard cycles. Outside these neighborhoods, the
two surfaces and their standard cycles agree, so this extends to a homeomorphism
$
f:\Sigma(G_1,\tau_1)\longrightarrow\Sigma(G_2,\tau_2)
$
preserving standard cycles.
\end{proof}
\begin{remark}\label{vertexflip is subset of mcg}

Let $\mathcal{V}_g:=\widehat{\mathcal{G}_g}/\sim_f$. By Lemma~\ref{flip eq vs mcg eq}, if two elements of $\mathcal{G}$ are
flip-equivalent, then the corresponding minimal filling pairs belong to
the same mapping class group orbit. Hence every flip equivalence
class is contained in a single mapping class group orbit.
Therefore, the map that sends each
flip equivalence class to the corresponding mapping class group orbit is
well-defined. In particular,
$
N(g)\leq |\mathcal{V}_g|.
$

\end{remark}
\subsection{Computing upper bound}
Here we estimate the upper bound of the cardinality of $\mathcal{V}_g$. By Lemma~\ref{bound of Gg}, we have
$
|\mathcal{G}_g|\leq 2^{g-1}(g-2)!.
$
For each $G\in\mathcal{G}_g$, a twist function is a $1$-cochain
$
\tau\in C^1(G;\mathbb{Z}/2\mathbb{Z}).
$
Since $G$ has $2(g-1)$ edges, we have
$
\left|C^1(G;\mathbb{Z}/2\mathbb{Z})\right|
=2^{2(g-1)}.
$
Therefore,
\[
\left|\left\{(G,\tau):
G\in\mathcal{G}_g,\,
\tau\in C^1(G;\mathbb{Z}/2\mathbb{Z})
\right\}\right|
\leq 2^{3g-3}(g-2)!.
\]

Now, for a fixed $G\in\mathcal{G}_g$, the condition $[\tau]=0$ is
equivalent to $\tau\in B^1(G;\mathbb{Z}/2\mathbb{Z})$. Moreover,
\[
\dim {B^1(G,\z/2\z)}
=
\dim C^0(G;\mathbb Z/2\mathbb Z)-1
=
|V(G)|-1
=
g-2.
\]
Consequently,
$
|B^1(G;\mathbb{Z}/2\mathbb{Z})|=2^{g-2}.
$ Therefore, the number of twist functions $\tau\in
C^1(G;\mathbb{Z}/2\mathbb{Z})$ representing a nontrivial cohomology
class is $2^{2(g-1)}-2^{(g-2)}$. Finally, we get,
\[
|\widehat{\mathcal{G}_g}|\le |\{(G,\tau): G\in \mathcal{G}_g;
\tau\in C^1(G,\z/2\z) ,
\left[\tau\right]\neq0\}\le2^{g-1}(g-2)!(2^{2(g-1)}-2^{(g-2)})
\]

Now, let $[(G,\tau)]\in\mathcal{V}_g$. Each flip-equivalence
class contains at most $2^{g-1}$ elements of
$\widehat{\mathcal{G}_g}$. On the other hand, in the above estimate for
$\left|\widehat{\mathcal{G}_g}\right|$, all these equivalent
representatives have been counted separately. Hence,
\[
|\mathcal{V}_g|
\leq
\frac{
2^{g-1}(g-2)!
\left(2^{2(g-1)}-2^{g-2}\right)
}{
2^{g-1}
}.
\]
By Remark~\ref{vertexflip is subset of mcg}, we have
$
N(g)\leq |\mathcal{V}_g|.
$
Consequently,
\[
N(g)
\leq
\left(2^{2(g-1)}-2^{g-2}\right)(g-2)!.
\]
\begin{remark}
In the above counting, we have not imposed the condition that
$\Sigma(G,\tau)$ has a single boundary component. Incorporating this
condition into the counting may lead to a sharper upper bound and could
potentially yield the asymptotic behavior of $N(g)$.
\end{remark}

\section{Length of filling pairs}
In this section, we study the hyperbolic length of minimal filling pairs on \(N_g\). We first determine the minimum value for the total length of minimal filling pairs. We then investigate the equality case and the corresponding subset of the moduli space, relating it to the mapping class group orbits of minimal filling pairs. Finally, we show that the same lower bound holds for arbitrary filling pairs, not only for minimal ones.

Let \(N_g\) be a hyperbolic non-orientable surface of genus \(g\geq 3\),
and let \(\mathrm{MF}(N_g)\) denote the set of all minimal filling pairs
on \(N_g\). Let \(\mathcal{M}_g\) be the moduli space of \(N_g\).
For \(X\in\mathcal{M}_g\) and
\((\alpha,\beta)\in\mathrm{MF}(N_g)\), define the length of
\((\alpha,\beta)\) with respect to \(X\) by
\[
\ell_X(\alpha,\beta)
=
\ell_X(\alpha)+\ell_X(\beta),
\]
where \(\ell_X(\alpha)\) and \(\ell_X(\beta)\) denote the lengths of the
geodesic representatives of \(\alpha\) and \(\beta\), respectively,
with respect to the hyperbolic metric of \(X\).

We now define a function
$
\mathcal{F}_g:\mathcal{M}_g\to\mathbb{R}
$
by
$
\mathcal{F}_g(X)
=
\min_{(\alpha,\beta)\in\mathcal{U}_g}
\ell_{X}(\alpha,\beta).
$
Thus, \(\mathcal{F}_g(X)\) is the minimum total length of a minimal
filling pair on \(N_g\) with respect to the hyperbolic metric \(X\).
The following lemma determines the minimum value of this function.
\begin{lemma}
Let \(g\geq 3\). Then, for every \(X\in\mathcal{M}_g\),
$
\mathcal{F}_g(X)\geq m_g,
$
where
\[
m_g=(4g-4)\operatorname{arccosh}
\left(
\sqrt{2}\cos\frac{\pi}{4g-4}
\right).
\]
\end{lemma}

\begin{proof} Let \(X\in\mathcal{M}_g\) and let \((\alpha,\beta)\) be a minimal filling pair on \(X\) whose length is minimum.
Since \((\alpha,\beta)\) is a minimal filling pair on \(N_g\), cutting \(X\) along \(\alpha\cup\beta\) yields a  polygon \(P\) with \(4g-4\) sides. Moreover,
$
\operatorname{Area}(P)=\operatorname{Area}(N_g)=2\pi(g-2).
$
By Bezdek~\cite{MR823098}, among all hyperbolic \(n\)-gons enclosing a
fixed area, the regular \(n\)-gon has the smallest perimeter. Hence,
the perimeter of \(P\) is the perimeter of the regular
\((4g-4)\)-gon of area $2\pi(g-2)$. By the Gauss--Bonnet theorem for
hyperbolic polygons, each interior angle of $P$ is \(\pi/2\).
Hence, using the standard trigonometric formula (see~\cite[Chapter~2]{MR1183224}), we get, \[ \operatorname{Per}(P) \geq (8g-8)\operatorname{arccosh} \left( \sqrt{2}\cos\frac{\pi}{4g-4} \right) =2m_g. \] Since every side of \(P\) occurs twice when \(P\) is glued back to \(X\), $ \operatorname{Per}(P) = 2\bigl(\ell_X(\alpha)+\ell_X(\beta)\bigr) = 2\mathcal{F}_g(X). $ Therefore, $ \mathcal{F}_g(X)\geq m_g. $
\end{proof}

\medskip

The lemma above determines the minimum possible value of
\(\mathcal{F}_g\), which naturally leads us to study the hyperbolic
structures for which this minimum is attained.We define
\[
\mathcal{B}_g
=
\left\{
X\in\mathcal{M}_g
\;\middle|\;
\mathcal{F}_g(X)=m_g
\right\}.
\]
Thus, \(\mathcal{B}_g\) consists of those hyperbolic structures on
\(N_g\) for which a minimal filling pair realizes the minimum value \(m_g\).

\begin{pro}\label{5.2}
    For $g\ge3$, $\mathcal{B}_g$ is finite and $|\mathcal{B}_g|=N(g)$.
\end{pro}
Before going to the proof of the theorem, we need to prove the following lemma.
\begin{lemma}\label{proper}
    $\mathcal{F}_g$ is a proper function.
\end{lemma}
\begin{proof}
Let $K\subset \mathbb{R}$ be a compact set. Since $K$ is compact, there exists
$M>0$ such that
$
K\subset [m_g,M].
$
Now, we will show that $\mathcal{F}_g^{-1}(K)$ is compact.

Suppose $X\in \mathcal{M}_g$ and let $c=\operatorname{inj}(X)$ be its
injectivity radius. Then there exists a simple closed geodesic
$\gamma$ on $X$ such that
$
\ell(\gamma)=2c .
$
By Lemma~\ref{collar lemma}, $\gamma$ admits an embedded collar of width $\omega_\gamma$ such that $\omega_\gamma\rightarrow\infty$ as $\ell_X(\gamma)\rightarrow0$.
Let $(\alpha,\beta)$ be a filling pair on $X$ realizing
$\mathcal{F}_g(X)$. Since $(\alpha,\beta)$ fills $X$, at least one of $\alpha$ or $\beta$ intersects $\gamma$.
Every geodesic arc crossing the collar of $\gamma$ has length at least
$\omega_\gamma$. Therefore,
$
\mathcal{F}_g(X)
=\ell(\alpha)+\ell(\beta)
\ge \omega_\gamma.
$

Since $\omega_\gamma\to\infty$ when $\ell_X(\gamma)\to0$, it follows that
$
\mathcal{F}_g(X)\to\infty$
whenever $
\operatorname{inj}(X)\to0.
$
Hence there exists $\varepsilon=\varepsilon(M)>0$ such that
$
\mathcal{F}_g(X)\le M$
implies
$\operatorname{inj}(X)\ge \varepsilon .
$
Therefore,
\[
\mathcal{F}_g^{-1}(K)
\subset
\{X\in\mathcal{M}_g:\operatorname{inj}(X)\ge \varepsilon\}.
\]
By \cite[Proposition 2.3]{injectivityradius}, the latter set is compact.
Since $\mathcal{F}_g$ is continuous, $\mathcal{F}_g^{-1}(K)$ is a closed
subset of a compact set, hence compact.
\end{proof}
\begin{lemma}\label{morse}
    $\mathcal{F}_g$ is a topological Morse function.
\end{lemma}
\begin{proof}
    The proof follows from a similar argument as in the proof of \cite[Theorem 1.3]{MR3342680}.
\end{proof}
\begin{lemma}\label{dd symbol lemma}
    Let $(\alpha,\beta)$ and $(\tilde{\alpha},\tilde{\beta})$ be two minimal filling pair on a hyperbolic surface $X\in\mathcal{B}_g$ each of length $m_g$. Then there exist a homeomorphism $\phi:X\rightarrow X$ such that $\phi(\alpha,\beta)=(\tilde{\alpha},\tilde{\beta})$.
\end{lemma}
\begin{proof}
The proof follows the similar argument as that of
\cite[Proposition 7.6]{MR4926623}. Let $(\alpha,\beta)$ and
$(\tilde{\alpha},\tilde{\beta})$ be two minimal filling pairs on
$X\in\mathcal{B}_g$, each consisting of curves of length
$m_g$.

Let $\Gamma$ be a NEC group such that
$
X=\mathbb{H}^2/\Gamma.
$
The filling pairs $(\alpha,\beta)$ and
$(\tilde{\alpha},\tilde{\beta})$ determine two tilings
$\mathcal{T}$ and $\mathcal{T}'$ of $\mathbb{H}^2$ by right-angled
regular $(4g-4)$-gons. Hence we obtain two equivariant tilings
$(\mathcal{T},\Gamma)$ and $(\mathcal{T}',\Gamma)$.

Let $\mathcal{C}_{\mathcal{T}}$ and $\mathcal{C}_{\mathcal{T}'}$
be the corresponding chamber systems, and let
$(\mathcal{D};m)$ and $(\mathcal{D}';m')$ be the associated
Delaney--Dress symbols.
Since each tile is a right-angled $(4g-4)$-gon, we have
\begin{align*}
    m_{02}(D)&=2=m'_{02}(D'),\\m_{12}(D)&=4=m'_{12}(D'),\\m_{01}(D)&=4g-4=m'_{01}(D')
\end{align*}
for all $D\in\mathcal{D}$ and $D'\in\mathcal{D}'$. Moreover,
$
|\mathcal{D}|=|\mathcal{D}'|=8g-8.
$
Therefore, the Delaney--Dress symbols
$(\mathcal{D};m)$ and $(\mathcal{D}';m')$ are isomorphic. By
Lemma~\ref{dd symbol}, it follows that the equivariant tilings
$(\mathcal{T},\Gamma)$ and $(\mathcal{T}',\Gamma)$ are
equivariantly equivalent. Hence, there exists a homeomorphism
$
\phi:\mathbb{H}^2\longrightarrow\mathbb{H}^2
$
such that
$\phi(\mathcal{T})=\mathcal{T}'.
$
Since $\phi\Gamma\phi^{-1}=\Gamma$, the map $\phi$ descends to a
homeomorphism
$
\tilde{\phi}:X\longrightarrow X.
$
By construction, $\tilde{\phi}$ sends the filling pair
$(\alpha,\beta)$ to $(\tilde{\alpha},\tilde{\beta})$.
\end{proof}
\medskip

\textbf{Proof of \propref{5.2}}
By Lemma~\ref{proper}, the function $\mathcal{F}_g$ is proper. Hence, 
the set
$
\mathcal{B}_g=\mathcal{F}_g^{-1}(m_g)
$
is compact.
By Lemma~\ref{morse}, $\mathcal{F}_g$ is a topological Morse function. Therefore,
$\mathcal{B}_g$ are isolated. Consequently, it is a compact
discrete subset of $\mathcal{M}_g$, and hence $\mathcal{B}_g$ is finite.

Now we define a map
$
\Phi:\mathcal{N}_g\longrightarrow\mathcal{B}_g.
$
To each mapping class group orbit of a minimal filling pair, we associate
the corresponding right-angled regular polygon with its side-pairing
data. More precisely, let
$
[(\alpha,\beta)]\in\mathcal{B}_g
$
be a mapping class group orbit of minimal filling pairs on $N_g$. Cutting
$N_g$ along $\alpha\cup\beta$ yields a hyperbolic
$(4g-4)$-gon.
We define $\Phi$ by replacing this polygon with the right-angled regular
hyperbolic $(4g-4)$-gon together with a side-pairing map determined by the
identifications of the sides arising from $\alpha$ and $\beta$. injectivity of $\Phi$ follows from Lemma~\ref{dd symbol lemma} and surjectivity is trivial.
Consequently,
\[
|\mathcal{B}_g|=|\mathcal{N}_g|=N(g).
\]
\qed

We now consider a generalized version of $\mathcal{F}_g$, where arbitrary
filling pairs are allowed instead of restricting to minimal filling pairs.
This allows us to study the shortest filling pair on a given hyperbolic
surface. Let $\mathcal{N}'_g$ denote the set of mapping class group orbits
of filling pairs on $N_g$. For $X\in\mathcal{M}_g$, define
\[
\mathcal{Y}_g:\mathcal{M}_g\longrightarrow\mathbb{R},
\qquad
\mathcal{Y}_g(X)
=
\min_{(\alpha,\beta)\in\mathcal{N}'_g}
\ell_X(\alpha,\beta).
\]

We next establish the minimum value for $\mathcal{Y}_g$ using the following
theorem of Sanki--Vadnere \cite{MR4278333}.

\begin{theorem}[Sanki--Vadnere]\label{sanki vadnere}
Suppose $P_i$'s are hyperbolic $2m_i$-gons, with $m_i\geq 2$, for
$i=1,2,\ldots,k$. Let $\mathcal{R}$ be a regular $N$-gon such that
\begin{enumerate}
    \item $N=4(1-k)+2\sum_{i=1}^k m_i$, and
    \item
    $
    \operatorname{area}(\mathcal{R})
    =
    \sum_{i=1}^k\operatorname{area}(P_i).
    $
\end{enumerate}
If $\mathcal{R}$ is not acute, then
$
\sum_{i=1}^k\operatorname{Per}(P_i)
\geq
\operatorname{Per}(\mathcal{R}).
$
\end{theorem}

The above theorem applied to the polygonal decomposition associated to a
filling pair gives the following lower bound.

\begin{pro}\label{Yg lower bound}
Let $g\geq 3$. Then, for every $X\in\mathcal{M}_g$,
$
\mathcal{Y}_g(X)\geq m_g.
$
\end{pro}

\begin{proof}
Let $(\alpha,\beta)$ be a filling pair on $X\in\mathcal{M}_g$, and let
$P_1,\ldots,P_k$ be the polygons in complementary decomposition
$
X\setminus(\alpha\cup\beta).
$
Suppose that each $P_i$ is a hyperbolic $2m_i$-gon. The graph
$\alpha\cup\beta$ has $g+k-2$ vertices and $2(g+k-2)$ edges. Hence,
counting the sides of the complementary polygons, we have
$
\sum_{i=1}^k 2m_i=4(g+k-2).
$
Let $\mathcal{R}$ be a regular $N$-gon, where
$
N=4(1-k)+\sum_{i=1}^k2m_i=4g-4.
$
Moreover,
\[
\operatorname{area}(\mathcal{R})
=
\sum_{i=1}^k\operatorname{area}(P_i)
=
\operatorname{area}(X)
=
2\pi(g-2).
\]
Thus $\mathcal{R}$ is the right-angled regular hyperbolic
$(4g-4)$-gon, and consequently
$
\operatorname{Per}(\mathcal{R})=2m_g.
$
Therefore, by Theorem~\ref{sanki vadnere},
\[
2\ell_X(\alpha,\beta)
=
\sum_{i=1}^k\operatorname{Per}(P_i)
\geq
\operatorname{Per}(\mathcal{R})
=
2m_g.
\]
Hence
$
\ell_X(\alpha,\beta)\geq m_g.
$
Since $(\alpha,\beta)$ is arbitrary, taking the minimum over all filling
pairs gives
$
\mathcal{Y}_g(X)\geq m_g.
$
\end{proof}
\bibliographystyle{plain}
\bibliography{ref}
\end{document}